\documentclass[12pt,reqno]{amsart}
\usepackage{diagrams}
\usepackage{amsfonts}
\usepackage{indentfirst}
\usepackage{graphicx}
\usepackage{amssymb}
\usepackage{color,xcolor}
\usepackage{graphicx,epstopdf}
\usepackage{epsfig}
\usepackage{subfig}
\usepackage{caption}
\usepackage{mathrsfs}
\usepackage{bbm}
\usepackage{chngpage}
\usepackage{cite}
\usepackage{multirow}
\usepackage{multicol}
\usepackage{dsfont}
\usepackage{bm}
\usepackage{amssymb,amsmath,amsthm,mathrsfs}%
\usepackage{exscale}
\usepackage{tikz-cd}
\usepackage{relsize}
\usepackage{amssymb}
\usepackage{hyperref}
\usepackage{url}
\usepackage{comment}
\usepackage[misc,geometry]{ifsym}

\allowdisplaybreaks

\hypersetup{hypertex=green,
            colorlinks=true,
            linkcolor=green,
            anchorcolor=green,
            citecolor=red}
\theoremstyle{definition}%{remark}{\plain}
\newtheorem{theorem}{Theorem}[section]

\newtheorem{example}{Example}

\newtheorem{remark}{Remark}[section]
\newtheorem{lemma}{Lemma}[section]

\numberwithin{equation}{section}%
\numberwithin{table}{section}%
\numberwithin{figure}{section}

\def\3bar{{|\hspace{-.02in}|\hspace{-.02in}|}}

\newcommand\grad{\operatorname{grad}}
\renewcommand\div{\operatorname{div}}

\newcommand\curl{\operatorname{curl}}

\newcommand\Span{\operatorname{span}}

\def\d{\text{d}}

\begin{document}
\title[]{Three families of  grad-div-conforming finite elements}


\author{Qian Zhang}
\email{go9563@wayne.edu}
\address{Department of Mathematics, Wayne State University, Detroit, MI 48202, USA. }

\author{Zhimin Zhang}
\email{zmzhang@csrc.ac.cn; zzhang@math.wayne.edu}
\address{Beijing Computational Science Research Center, Beijing, China; Department of Mathematics, Wayne State University, Detroit, MI 48202, USA}
\thanks{This work is supported in part by the National Natural Science Foundation of China grants NSFC 11871092 and NSAF U1930402.}


\date{\today}
\begin{abstract} 
Several smooth finite element de Rham complexes are constructed in three-dimensional space, which yield three families of grad-div conforming finite elements. The simplest element has only 8  degrees of freedom (DOFs) for a tetrahedron and 14 DOFs for a cuboid. These elements naturally lead to conforming approximations to quad-div problems. 
Numerical experiments for each family validate the correctness and efficiency of the elements for solving the quad-div problem. 

\end{abstract}	
\maketitle
\section{Introduction}
We are concerned in this paper with the grad-div or $H(\grad\div)$ finite  elements used for conforming discretizations of problems involving a quad-div operator (sometimes referred to as a fourth-order div operator). 
The quad-div operator appears in linear elasticity \cite{altan1992structure,mindlin1963microstructure,mindlin1965second}, where the integration of $\left(\nabla(\nabla\cdot \bm u)\right)^2$ represents the shear strain energy with the displacement of the elasticity body $\bm u$.
 Moreover, the quad-div operator can be written as $\big((\nabla\cdot)^*\circ (\nabla\cdot)\big)^*\circ (\nabla\cdot)^*\circ (\nabla\cdot)$ which is one of the fourth-order operators of the formulation $(D^*\circ D)^*\circ D^*\circ D$. The biharmonic operator $\Delta^2$ and the quad-curl operator $(\nabla\times)^4$ are two well-known fundamental fourth-order operators of this formulation, which have been studied extensively. Many numerical methods have been proposed for problems involving those two fourth-order operators. We refer to \cite{WZZelement,quad-curl-eig-posterior,HZZcurlcurl2D,zhang2009family,argyris1968tuba,vzenivsek1973polynomial} for conforming finite element methods and \cite{morley1968triangular,Zheng2011A,Sun2016A,Qingguo2012A,Brenner2017Hodge,quadcurlWG, Zhang2018M2NA,Chen2018Analysis164, Zhang2018Regular162,WangC2019Anew101,BrennerSC2019Multigrid100} for other methods. However, unlike the biharmonic operator and the quad-curl operator, very limited work has been done for problems involving the quad-div operator \cite{fan2019mixed}. In this paper, we will propose grad-div conforming  elements which can lead to conforming approximations of the quad-div problem.
 
We apply discrete de Rham complexes to construct three families of grad-div conforming  finite elements in three space dimensions (3D). The discrete de Rham complex with appropriate smoothness has been an important and useful tool in designing finite elements and analyzing numerical schemes, c.f., \cite{arnold2018finite, arnold2010finite, arnold2006finite, hiptmair1999canonical, neilan2015discrete,christiansen2018nodal}. Based on the de Rham complex with minimal smoothness, various well-known finite elements for computational electromagnetism or diffusion problems have been arranged in the finite element periodic table \cite{arnold2014periodic}. Motivated by problems in fluid and solid mechanics, there is a growing interest in constructing finite element de Rham complexes with enhanced smoothness, sometimes also referred to as Stokes complexes \cite{falk2013stokes,christiansen2016generalized,HZZcurlcurl2D}. In this paper, for the conforming discretization of the quad-div problem, we will consider another variant of the de Rham complex, i.e.,
\begin{align}\label{3D:quad-div}
\begin{diagram}
\mathbb{R} & \rTo^{\subset} & H^{1}(\Omega) & \rTo^{\nabla} & H(\curl; \Omega) & \rTo^{\nabla\times} & H(\grad \div;\Omega)& \rTo^{\nabla\cdot} & H^1(\Omega)  & \rTo^{} & 0,
\end{diagram}
\end{align}
where  $\Omega$ is a bounded Lipschitz domain in $\mathbb{R}^{3}$. The precise definition of the notations used in the above complex is given in the subsequent sections. For simplicity of presentation, throughout this paper,  we will assume that $\Omega$ is contractible. Then the exactness of \eqref{3D:quad-div} follows from standard results in, e.g., \cite{arnold2018finite}. The two dimensional (2D) version of the complex \eqref{3D:quad-div} is
\begin{align}\label{2D:quad-div}
\begin{diagram}
\mathbb{R} & \rTo^{\subset} & H^{1}(\Omega) & \rTo^{\bm\nabla\times} & H(\grad \div;\Omega)& \rTo^{\nabla\cdot} & H^1(\Omega)  & \rTo^{} & 0,
\end{diagram}
\end{align}
where $\bm \nabla\times=(\partial_{x_2}u, -\partial_{x_1}u)^T$. 
If we rotate the complex \eqref{2D:quad-div} by $\frac{\pi}{2},$ we will get the following complex
\begin{align}
\begin{diagram}
\mathbb{R} & \rTo^{\subset} & H^{1}(\Omega) & \rTo^{{\nabla}} & H(\curl^2;\Omega)& \rTo^{\nabla\times} & H^1(\Omega)  & \rTo^{} & 0
\end{diagram}
\end{align}
with $H(\curl^2;\Omega):=\{\bm u\in \bm L^2(\Omega): \nabla\times \bm u\in L^2(\Omega) \text{ and } \bm \nabla\times\nabla\times \bm u\in \bm L^2(\Omega)\}$,
 which has been studied in \cite{HZZcurlcurl2D}.
 Therefore in this paper, we will only focus on the complex \eqref{3D:quad-div} to construct 3D grad-div elements.
 
Our new finite elements fit into a subcomplex of \eqref{3D:quad-div}:
\begin{align}\label{discrete-complex}
\begin{diagram}
\mathbb{R} & \rTo^{\subset} & \Sigma_h & \rTo^{\nabla} & V_{h} & \rTo^{\nabla\times} & W_h & \rTo^{\nabla\cdot} & \Sigma^{+}_h  & \rTo^{} & 0.
\end{diagram}
\end{align}
A starting point is to take $\Sigma_h$ as $C^0$ Lagrange finite element spaces.  This leads to a natural choice of $V_h$, which is the first family of N\'ed\'elec elements. In addition, we choose Lagrange elements enriched with an interior bubble for $\Sigma_h^+$. The space $W_h\subset H(\grad \div;\Omega)$ is hence obtained as the curl of $V_h$ plus a complementary part, mapped onto $\Sigma_h^+$ by $\div$. Different orders of $\Sigma_h$  can yield different versions of $W_h$.  
Among the three versions of $W_h$ which we will construct in this paper, the simplest element has only 8 DOFs for a tetrahedron and 14 DOFs for a cuboid. The space $W_h$ can be utilized as a conforming finite element space for solving the quad-div problem.

The tool of discrete complexes makes it possible to reach the goal of this paper, i.e., constructing grad-curl conforming elements. Moreover, by this tool, we also fit the quad-div problem and its conforming finite element approximations into the framework of the finite element exterior calculus (FEEC) \cite{arnold2018finite,arnold2006finite} and hence enable the use of various tools from FEEC for the numerical analysis. For instance, we construct interpolation operators that commute with the differential operators. After that, the convergence result follows from a standard argument. 

To validate the newly proposed elements, we fit them in the  theoretical framework of the conforming finite element method and obtain the approximation property of the numerical solution. Moreover, we carry out a numerical experiment, which validates our theoretical results.

The remaining part of the paper is organized as follows. In Section 2, we present notations and some basic facts from homological algebra.
In Section 3, we define shape functions which form local exact sequences and prove their properties. 
In Section 4, we construct three families of grad-curl conforming finite elements on tetrahedra and develop some theoretical results for them. In Section 5, we introduce the counterparts for cuboids. In Section 6, we present a regularity estimate for the quad-div problem and apply the new elements to the problem. In Section 7, we provide numerical examples to verify the correctness and efficiency of our methods. Finally, concluding remarks and future work are given in Section 8.

\section{Preliminaries}

\subsection{Notations}

Unless otherwise specified, throughout the paper we assume $\Omega\in\mathbb{R}^3$ is a contractible Lipschitz domain. We adopt conventional notations for Sobolev spaces such as $H^m(D)$ or $H^m_0(D)$ on a contractible sub-domain $D\subset\Omega$ furnished with the norm $\left\|\cdot\right\|_{m,D}$ and the semi-norm $\left|\cdot\right|_{m,D}$. In the case of $m=0$, the space $H^{0}(D)$ coincides with $L^2(D)$ which is equipped with the inner product $(\cdot,\cdot)_D$ and the norm $\left\|\cdot\right\|_D$. When $D=\Omega$, we drop the subscript $D$. We use  $\bm H^{m}(D)$   and ${\bm L}^2(D)$ to denote the vector-valued Sobolev spaces  $\left(H^{m}(D)\right)^3$ and $\left(L^2(D)\right)^3$.

%Let ${\bm u}=(u_1, u_2,u_3)^T$ and ${\bm w}=(w_1, w_2, w_3)^T$, where the superscript $T$ denotes the transpose,
%then ${\bm u} \times {\bm w} = (u_2w_3-w_2u_3, w_1u_3-u_1w_3, u_1w_2-w_1u_2)^T$ and 
%$\nabla \times {\bm u} = (\partial_{x_2}u_{3}- \partial_{x_3}u_{2}, \partial_{x_3}u_{1} -  \partial_{x_1}u_{3}, \partial_{x_1}u_{2} - \partial_{x_2}u_{1} )^T$.
%For convenience, here and hereinafter we abbreviate the partial differential operators $\frac{\partial }{\partial x_i}$
%to $\partial_{ x_i}$.
%We denote $(\nabla\times)^2\bm u=\nabla\times\nabla\times\bm u$.

In addition to the standard Sobolev spaces, we also define
\begin{align*}
	H(\text{curl};D)&:=\{\bm u \in {\bm L}^2(D):\; \nabla \times \bm u \in \bm L^2(D)\},\\
	H(\text{div};D)&:=\{\bm u \in {\bm L}^2(D):\; \nabla \cdot  \bm u \in L^2(D)\},\\
	H(\grad \text{div};D)&:=\{\bm u \in {\bm L}^2(D):\; \nabla \cdot  \bm u \in H^1(D)\}.
\end{align*}
%with scalar products and norms are defined by
%	\[(\bm u,\bm v)_{H(\curl^s;D)}=(\bm u,\bm v)+\sum_{j=1}^s((\nabla\times)^j\bm u,(\nabla\times)^j \bm v)\]
%	and
%	\[\left\|\bm u\right\|_{H(\curl^s;D)}=\sqrt{(\bm u,\bm u)_{H(\curl^s;D)}}.\]
%The spaces $H_0(\text{curl}^s;D)(s=1,\;2)$ are defined as follows:
%\begin{align*}
%	&H_0(\text{curl};D):=\{\bm u \in H(\text{curl};D):\;{\bm n}\times\bm u=0\; \text{on}\ \partial D\},\\
%	&H_0(\text{curl}^2;D):=\{\bm u \in H(\text{curl}^2;D):\;{\bm n}\times\bm u=0\; \text{and}\; \nabla\times \bm u=0\;\; \text{on}\ \partial D\}.
%\end{align*}
%The space of $\bm L^2(D)$ functions with square-integrable divergence is denoted by $H(\td;D)$ and defined by
%\[H(\text{div};D) :=\{\bm u\in {\bm L}^2(D):\; \nabla\cdot \bm u\in L^2(D)\}\]
%with the associated inner product
%$(\bm u,\bm v)_{H(\td;D)}=(\bm u,\bm v)+(\nabla\cdot \bm u,\nabla\cdot \bm v)$ and the norm
%$\left\|\bm u\right\|_{H(\td;D)}=\sqrt{(\bm u,\bm u)_{H(\td;D)}}$.

%We use $Q_{i,j,k}$ to represent the polynomials in three variables $(x_1,x_2,x_3)$ whose maximum degree  are respectively $i$ in $x_1$, $j$ in $x_2$, and $k$ in $x_3$. For simplicity, we drop   subscripts $i$ and $j$ when $i=j=k$. 
For a subdomain $D$, a face $f$, or an edge $e$, we use $P_k$ to represent the space of 
polynomials with degree at most $k$, and $\tilde P_k$, the space of homogenous polynomials of degree $k$. The corresponding sets of vector polynomials are denoted as $\bm P_k=\left(P_k(D) \right)^3$ and $\tilde{\bm P}_k=(\tilde P_k(D) )^3$, respectively. 
We also define
\begin{align*}
&\mathcal R_k=\bm P_{k-1}\oplus \mathcal S_k \text{\ with\ }\mathcal S_k=\{{\bm p}\in \tilde{\bm P_k}\big| \ \bm x\cdot \bm p=0\},
%\mathcal D_k=\bm P_{k-1}\oplus  \bm x& \tilde P_{k-1} \text{\ with\ } \tilde P_{k-1}\text{\ is the space of homogeneous polynomials of degree } k-1 .
\end{align*}
%The dimensions of these special spaces of vector polynomials  are 
whose dimension is
\begin{align*}
&\dim{\mathcal R_k}=\frac{k(k+2)(k+3)}{2}.
%&\dim{\mathcal D_k}=\frac{k(k+1)(k+3)}{2}.
\end{align*}
%For the space $\bm P_k$, we have the following decomposition \cite{da2018lowest}
%\begin{align}
%&\bm P_k=\nabla P_{k+1} \oplus \bm x\times \bm P_{k-1},\label{Pkdecomp1}\\
%&\bm P_k=\nabla\times \mathcal R_{k+1} \oplus  \bm x P_{k-1}\label{Pkdecomp2}.
%\end{align}
%The dimension of $ \bm x\times \bm P_{k-1}$ is $\dim{\bm P_k}-\dim{P_{k+1}}+1$ and the dimension of $\nabla\times \mathcal R_{k+1}$ is $\dim{\bm P_k}-\dim{P_{k-1}}$.

We use $Q_{i,j,k}(D)$ to denote the polynomials with three variables $(x_1, x_2, x_3)$ where the maximal degree is $i$ in $x_1$, $j$ in $x_2$, and $k$ in $x_3$. For simplicity, we drop the subscripts $i$ and $j$ when $i=j=k$.  Similarly, we use $Q_{i,j}(f)$ to denote such polynomial spaces in 2D. 

%\cite{piolamapping} 

Let \,$\mathcal{T}_h\,$ be a partition of the domain $\Omega$
consisting of tetrahedra or cuboids. We denote $h_K$ as the diameter of an element $K \in
\mathcal{T}_h$ and $h$ as the mesh size of $\mathcal {T}_h$.
We adopt the following Piola mapping to relate the finite element function $\bm u$ on a general element $K$ to a function $\hat{\bm u}$ on the reference element $\hat K$ (the tetrahedron with vertices $(0,0,0)$, $(1,0,0)$, $(0,1,0)$, and $(0,0,1)$ or the cube with vertices $(0,0,0)$, $(1,0,0)$, $(0,1,0)$, $(1,1,0)$, $(0,0,1)$, $(1,0,1)$, $(0,1,1)$, and $(1,1,1)$):
\begin{align}
\bm u \circ F_K = \frac{B_K}{\det(B_K)} \hat{\bm u},\label{mapping-u}
\end{align}
where the affine mapping
\begin{align}\label{mapping-domain}
F_K(\bm x)= B_K\hat{\bm x}+ \bm b_K.
\end{align}
By a simple computation, we have
\begin{align}
(\nabla\cdot\bm u) \circ F_K &= \frac{1}{\det(B_K)} \hat{\nabla}\cdot\hat{ \bm u},\label{mapping-curlu}\\
\bm n\circ F_K&= \frac{B_K^{-T} \hat {\bm n}}{\left\|B_K^{-T} \hat {\bm n}\right\|}.\label{n}
\end{align}

We use $C$ to denote a generic positive $h$-independent constant. 

\subsection{Basic facts from homological algebra}
We review some basic facts from homological algebra. For further details, we refer to, for example,  \cite{arnold2018finite}. A differential complex is a sequence of spaces $V^{i}$ and operators $d^{i}$:
\begin{align}\label{general-complex}
\begin{diagram}
0 & \rTo & V^{1} & \rTo^{d^{1}} &V^{2}&\rTo^{d^{2}} & \cdots  &  \rTo^{d^{n-1}} &V^{n} & \rTo^{d^n} & 0,
\end{diagram}
\end{align}
with $d^{i+1}d^{i}=0$ for $i=1, 2, \cdots, n-1$. Denote $\mathcal{N}(d^{i})$ as the kernel space of the operator $d^{i}$ in $V^{i}$ and $\mathcal{R}(d^{i})$ as the image of the operator $d^{i}$ in $V^{i+1}$. Due to the definition of a complex, the inclusion $\mathcal{N}(d^{i})\subset \mathcal{R}(d^{i-1})$ holds for each $i\geq 2$. Furthermore, if $\mathcal{N}(d^{i})= \mathcal{R}(d^{i-1})$, we say that the complex \eqref{general-complex} is exact at $V^{i}$. At the two ends of the sequence, the complex is exact at $V^{1}$ if $d^{1}$ is injective (with trivial kernel), and is exact at $V^{n}$ if $d^{n}$ is surjective (with trivial cokernel).  The complex \eqref{general-complex} is referred to as exact if it is exact at all the spaces $V^{i}$. If each space in \eqref{general-complex} has finite dimensions, then a necessary (but not sufficient) condition for the exactness of \eqref{general-complex} is the following dimension condition:
$$
\sum_{i=1}^{n} (-1)^{i}\dim (V^{i})=0. 
$$

%We choose the Lagrange finite space $\Sigma_h$ of order $r$ for the first $H^1(\Omega)$ space in the complex \eqref{3D:quad-div}.

\section{Local spaces and polynomial complexes}

To define a finite element space, we must supply, for each element $K\in\mathcal{T}_h$, 
\begin{itemize}
	\item shape functions;
	\item DOFs to guarantee appropriate continuity.
\end{itemize}
In this section, we will specify shape functions for each space involved in the complex \eqref{discrete-complex}. 
The local complex of function spaces on each $K\in \mathcal{T}_h$ for \eqref{discrete-complex} is denoted as follows: 
\begin{align}\label{local-complex}
\begin{diagram}
\mathbb{R} & \rTo^{\subset} & \Sigma^{r}_h(K) & \rTo^{\nabla} &V^{r}_{h}(K) & \rTo^{\nabla\times} & W_h^{r-1,k}(K)& \rTo^{\nabla\cdot} & \Sigma_h^{+,k-1}(K)   & \rTo^{} & 0.
\end{diagram}
\end{align}
In addition to the complex \eqref{discrete-complex}, we introduce the complex \eqref{local-complex} with two parameters $r$ and $k$ to specify degrees of spaces, which lead to several versions of complexes.

%Denote $\Sigma(u)=M_{v}(u)\cup M_{e}(u)\cup M_{K}(u)$.
%We now specify the shape function spaces $\Sigma_h(K)$, $V_h(K)$, and $W_h(K)$.
We let $\Sigma^{r}_h(K)$ be $P_{r}(K)$ for a tetrahedral element or $Q_{r}(K)$ for a cuboid element, and 
 let $V^{r}_h(K)$ be $\mathcal R_{r}(K)$ for a tetrahedral element or $Q_{r-1,r,r}(K)\times Q_{r,r-1,r}(K)\times Q_{r,r,r-1}(K)$ for a cuboid element. 
For a tetrahedral element $K$, we set
$$
\Sigma_h^{+,k-1}(K)=
\begin{cases}
	\Sigma_h^{k-1}(K),& k\geq 5,\\
	\Sigma_h^{k-1}(K)\oplus \Span\{B_t\},& k=2,3,4,
	\end{cases}
$$
where $B_t=\lambda_1\lambda_2\lambda_3\lambda_4$ with the barycentric coordinate $\lambda_i$. For a cuboid element $K$, we set
$$
\Sigma_h^{+,k-1}(K)=
\begin{cases}
	\Sigma_h^{k-1}(K),& k\geq 3,\\
	\Sigma_h^{k-1}(K)\oplus \Span\{B_c\},& k=2,
	\end{cases}
$$
where $B_c=\left(x-x_l\right)\left(x-x_r\right)\left(y-y_f\right)\left(y-y_b\right)\left(z-z_d\right)\left(z-z_u\right)$ with the element $K=(x_l,x_r)\times(y_f,y_b)\times (z_d,z_u)$. 
We define
\begin{align}\label{Vh2}
	W_h^{r-1,k}(K)=\nabla\times V^{r}_h(K)\oplus \mathfrak{p}\Sigma^{+,k-1}_h(K),
\end{align}
with $\mathfrak{p}: C^{\infty}(\mathbb{R}^{3})\mapsto \left [C^{\infty}(\mathbb{R}^{3})\right ]^{3}$ is an operator which maps a scalar function to a vector field:
$$
\mathfrak{p} u:=\int_{0}^{1}t^2 \bm{x}u(t\bm x)\, dt,
$$
where
$
\bm{x}:=(x_{1}, x_{2}, x_3)^T.
$
%We note $\bm{x}^{\perp}\cdot \bm{x}=0$. 
As a special case of the Poincar\'{e} operators (c.f. \cite{hiptmair1999canonical, christiansen2016generalized}), $\mathfrak{p}$ has the following properties:
\begin{itemize}
\item the null-homotopy identity
\begin{equation}\label{null-homotopy}
\nabla\cdot \mathfrak{p}u=u,~ \forall u\in C^{\infty}(\mathbb R^3);
\end{equation}
\item polynomial preserving property: if $u\in {P}_{r}(\mathbb{R}^{3})$, then $\mathfrak{p}u\in \bm {P}_{r+1}(\mathbb{R}^{3})$.
\end{itemize}

The above two properties can be obtained via simple calculations. By the null-homotopy identity \eqref{null-homotopy}, the right hand side of \eqref{Vh2} is a direct sum. 
\begin{lemma}
The local sequence \eqref{local-complex} is an exact complex. 
\end{lemma}
\begin{proof}
It's obvious that $\nabla\Sigma_h^r(K)\subset V_h^r(K)$. Since $W_h^{r-1, k}(K)=\nabla\times V_h^{r}(K)+\mathfrak p \Sigma_h^{+,k-1}(K)$ and the null-homotopy identity \eqref{null-homotopy},  we have $\nabla\times V^{r}_h(K)\subseteq W^{r-1, k}_h(K)$ and $\nabla\cdot W^{r-1, k}_h(K)= \Sigma_h^{+,k-1}(K)$. This shows that  \eqref{local-complex} is a complex. It remains to show the exactness. It's easy to check the exactness at $\Sigma_h^r(K)$ and $V_h^r(K)$. We now show that, for any $\bm v_h\in W^{r-1, k}_h(K)$ for which $\nabla\cdot\bm v_h=0$, there exists a $\bm u_h\in V^{r}_h(K)$ s.t. $\bm v_h=\nabla \times \bm u_h.$
Since $\bm v_h\in W^{r-1, k}_h$, we have $\bm v_h=\nabla\times \bm u_h+\mathfrak p w_h$ with $\bm u_h\in V^{r}_h(K)$ and $w_h\in \Sigma^{+,k-1}_h(K)$. By the null-homotopy identity \eqref{null-homotopy} again, $0=\nabla\cdot\bm v_h=w_h$. Therefore, $\bm v_h=\nabla\times \bm u_h.$ To prove the exactness at $\Sigma_h^{+,k-1}$, we only need to show the div operator $\nabla\cdot: W^{r-1, k}_h(K) \to \Sigma^{+,k-1}_h(K)$ is surjective. It is surjective since $\nabla\times W^{r-1, k}_h(K)=\Sigma^{+,k-1}_h(K)$. 
 \end{proof}

In the following lemma, we show that $W_h^{r-1,k}(K)$ contains some polynomial subspaces. It plays an essential role in analyzing the approximation properties of the finite element space $W_h^{r-1,k}$.

\begin{lemma}\label{Vh} The inclusion $\bm P_{r-1}(K)\subseteq W^{r-1, k}_h(K)$ holds when $r\leq k+1$. 
%Moreover,
%\begin{align*}
%	&V_h^{k-1,k}(K)=\mathcal R_k(K) \text{ when  $K$ is a tetradedron},\\
%	&V_h^{k,k}(K)=\bm P_k(K) \text{ when  $K$ is a tetradedron},\\
%	&V_h^{k,k}(K)=Q_{k,k-1,k-1}(K)\times Q_{k-1,k,k}(K)\times Q_{k-1,k-1,k}(K) \text{ when  $K$ is a cuboid}.
%\end{align*}
\end{lemma}
\begin{proof}
	We claim that 
		\begin{align}\label{dcmp_Pr}
		\bm P_{r-1}(K)=\nabla \times \bm P_{r}(K)\oplus \mathfrak{p} P_{r-2}(K).
	\end{align}
	In fact, by the polynomial preserving property of $\mathfrak{p}$, $\nabla\times \bm P_{r}(K)\oplus\mathfrak{p} P_{r-2}(K)\subseteq \bm P_{r-1}(K)$. To show \eqref{dcmp_Pr}, it remains to show 
	the two sides of \eqref{dcmp_Pr} have the same dimension.	By the null-homotopy identity \eqref{null-homotopy}, the right hand side is a direct sum. Therefore, 
	\begin{align*}
		\dim & [\nabla \times \bm P_{r}(K)\oplus \mathfrak{p} P_{r-2}(K)]=\dim  \nabla \times \bm P_{r}(K)+\dim P_{r-2}(K)\\
		&=\dim [\bm P_{r}(K)\slash \nabla P_{r+1}(K)]+\dim P_{r-2}(K)\\
		&=\dim \bm P_{r}(K) - \dim P_{r+1}(K)+1+\dim P_{r-2}(K),
	\end{align*}
	which is exactly the dimension of $\bm P_{r-1}(K)$. 
Combining \eqref{dcmp_Pr} and the fact that $ P_{r-2}(K)\subseteq \Sigma^{+,k-1}_h(K)$, we get $\bm P_{r-1}(K)\subseteq \nabla\times \bm P_r(K)\oplus \mathfrak{p}\Sigma^{+,k-1}_h(K)\subseteq W^{r-1, k}_h(K)$.
\end{proof}

%With all the preparations, we are now in a position to constructing and analyzing various families of grad-div conforming finite elements and complexes. For simplicity of presentation, we focus on analyzing the tetrahedral elements and only present the definition of the cuboid elements in Section 5. 

We are now ready to construct grad-div conforming finite elements and complexes. We focus on analyzing tetrahedral elements and only present the definition of cuboid elements in Section 5.

\section{Three families of grad-div conforming elements  on tetrahedra}\label{sec:tetrahedral}

The global discrete complex with specified degree for each space is given by
\begin{align}\label{discrete-complex-r-k}
\begin{diagram}
 \mathbb{R} & \rTo^{\subset} & \Sigma_h^{r} & \rTo^{\nabla} & V_{h}^{r} & \rTo^{\nabla\times} & W_h^{r-1,k}& \rTo^{\nabla\cdot} & \Sigma_h^{+,k-1}  & \rTo^{} & 0.
\end{diagram}
\end{align}
In this section, we construct grad-div conforming finite  elements and complexes on tetrahedra.  Assigning $r=k-1$, $k$, and $k+1$ in \eqref{discrete-complex-r-k} leads to three versions of grad-div conforming element spaces $W_h^{k-2,k}$, $W_h^{k-1,k}$, and $W_h^{k,k}$, for which, Fig. \ref{fig:third-family} demonstrate the case $k=2$.

\subsection{Degrees of freedom and global finite element spaces}\label{sec:dofs}
We define DOFs for the spaces in \eqref{discrete-complex-r-k}.

The DOFs for the Lagrange element $\Sigma^{r}_{h}$ can be given as follows.
\begin{itemize}
\item Vertex DOFs $ M_{v}({u})${\color{blue}:}%at all the vertices  $v_{i}$ of $K$:
$$
M_{v}(u)=\left\{u\left({v}_{i}\right) \text{ for all vertices $v_i$}\right\}.
$$
\item Edge DOFs $M_{e}(u)${\color{blue}:} %on all the edges $e_{i}$ of $K$:
\begin{align*} M_{e}(u)=\left\{\frac{1}{\operatorname{length}(e_i)} \int_{e_i} u v \mathrm{d} s\text { for all } v \in P_{r-2}(e_i) \text { and for all edges }e_i\right\}.
\end{align*}
 \item Face DOFs $M_{f}(u)${\color{blue}:} % on all the faces $f_{i}$ of $K$:
\begin{align*} M_{f}(u)=\left\{\frac{1}{\operatorname{area}(f_i)} \int_{f_i} u v \mathrm{d} A\text { for all } v \in P_{r-3}(f_i) \text { and for all faces }f_i\right\}.
\end{align*}
\item Interior DOFs $M_{K}(u)$: 
$$M_{K}(u)=\left\{\frac{1}{\operatorname{volume}(K_i)} \int_{K_i} u v \mathrm{d} V \text{ for all } v \in P_{r-4}(K_i) \text{ and for all elements } K_i 
\right\}.$$
\end{itemize}
For $ u \in H^{3/2+\delta}(\Omega)$ with $\delta >0$,  we can
define an $H^1$ interpolation operator $\pi_h: H^{3/2+\delta}(\Omega)\rightarrow \Sigma_h^{r}$ by the above DOFs s.t.
 \begin{eqnarray}\label{def-inte-H1}
& M_v( u-\pi_h u)=\{0\},\  M_e(u-\pi_hu)=\{0\},\nonumber\\
& M_f(u-\pi_hu)=\{0\}, \text{ and}\  M_K( u-\pi_h u)=\{0\}.
\end{eqnarray}

The DOFs for $\Sigma_{h}^{+,k-1}$ can be given similarly, with only one additional interior integration DOF on $K$ to deal with the interior bubble function. We denote $\tilde\pi_h$ as the $H^1$ interpolation operator to $\Sigma_{h}^{+,k-1}$ by these DOFs.

%We denote $B_K$ as the matrix in the affine mapping
%\begin{align}\label{mapping-domain}
%F_K(\bm x)= B_K\hat{\bm x}+ \bm b_K
%\end{align}
% which maps the reference element $\hat K$ to a general element $K$.
%
%
% Note that $\bm u$ and $\hat{\bm u}$ are related by 
%$$\bm u \circ F_K = \det(B_K)^{-1}B_K \hat{\bm u}.$$
 We choose the space $V_h^{r}$ as the first family of N\'ed\'elec elements, which has the following DOFs:
\begin{itemize}
\item Edge DOFs $\bm M_{e}(\bm u)$ (with a unit tangential vector $\bm \tau_i$):  %on all the edges $e_{i}$ of $K$ 
\begin{align*} \bm M_{e}(\bm u)=\left\{\int_{e_i} \bm u\cdot \bm \tau_i v \mathrm{d} s\text { for all } v \in P_{r-1}(e_i) \text { and for all edges }e_i\right\}.
\end{align*}
 \item Face DOFs $\bm M_{f}(\bm u)$ (with a unit normal vector $\bm n_i$): %on all the faces $f_{i}$ of $K$ 
\begin{align*}\bm M_{f}(\bm u)=&
\left\{\frac{1}{\operatorname{area}(f_i)} \int_{f_i} \bm u \bm v \mathrm{d} A\text { for all } \bm v=B_K \hat {\bm v}, \hat {\bm v} \in \bm P_{r-2}(\hat f_i), \hat{\bm v}\cdot \hat{\bm n}_i=0\right. \\
&\text{\quad\quad\quad and for all faces }f_i\Big\}.
\end{align*}
\item Interior DOFs $\bm M_{K}(\bm u)$: 
\begin{align*}\bm M_{K}(\bm u)=&\left\{ \int_{K_i} \bm u \bm v \mathrm{d} V \text{ for all } \bm v =\frac{B_{K_i}}{\det(B_{K_i})}\hat{\bm v}\text{ with }\hat{\bm v}\in \bm P_{r-3}(\hat K_i)\right.\\
&\text{  \quad\quad\quad and for all elements } K_i\Big\} .
\end{align*}
\end{itemize}
Assuming that $\bm u\in \bm H^{1/2+\delta}(\Omega)$  and $\nabla\times\bm u\in L^p(\Omega)$ with $\delta\geq 0$ and $p>2$. By the above DOFs, we 
define an $H(\curl)$ interpolation operator $\bm r_h$ which maps to $V_h^r$ and satisfies
 \begin{eqnarray}\label{def-inte-H1}
 \bm  M_e(\bm u-\bm r_h\bm u)=\{0\}, 
 \bm M_f(\bm u-\bm r_h\bm u)=\{0\}, \text{ and}\  \bm M_K( \bm u-\bm r_h \bm u)=\{0\}.
\end{eqnarray}

We now equip the space $W^{r-1,k}_h$ with the following DOFs:
\begin{itemize}
		\item Vertex DOFs $\bm M_{ {v}}({\bm u})$ at all vertices $ {v}_{i}$ of each $K$:
	\begin{equation}\label{tridef1-1}
	\bm M_{ {v}}({\bm u})=\left\{( \nabla\cdot {\bm u})(  v_{i}),\; i=1,\;2\;,\cdots,4\right\}.
	\end{equation}
	\item Edge DOFs $\bm M_{ {e}}(  {\bm u})$ at all edges $ {e}_i$ of each ${K}$:
	\begin{align}
		 \bm M_{ {e}}(  {\bm u})=	\left\{\frac{1}{\operatorname{length}(e_i)}\int_{e_i}\nabla\cdot{\bm u}q\d s,\ \forall  {q}\in P_{k-3}( {e}_i), i=1,2,\cdots,6\right\}.\label{tridef1-2}
	\end{align}	
	\item Face DOFs $\bm M_{f}({\bm u})$ at all faces ${f}_i$ of each ${K}$ (with the unit  normal vector $ {\bm n}_i$):
	\begin{align}
		 \bm M_{f}({\bm u})=&\left\{\frac{1}{\operatorname{area}(f_i)}\int_{f_i}\nabla\cdot{\bm u}q\d A,\ \forall  {q}\in P_{k-4}( {f}_i), i=1,2,\cdots,4\right\}\nonumber\\
		 &\cup \left\{\int_{f_i}{\bm u}\cdot \bm n_iq\d A,\ \forall  {q}\in P_{r-1}( {f}_i), i=1,2,\cdots,4\right\},
		 \label{tridef1-3}
	\end{align}	
    \item Interior DOFs $\bm M_{ {K}}(  {\bm u})$ for each element $K$:
	\begin{align}\label{tridef1-4}
	&\bm M_{ {K}}(  {\bm u})=\left\{\int_{ {K}}  {\bm u}\cdot  {\bm q}\mathrm \d V,\ \forall \bm q=B_{K}^{-T}\hat{ \bm q},\ \hat{\bm q} \in  \mathcal{D} \right\},
	%\label{def5}
	%&\bm M_{ {K}}(  {\bm u})=\left\{\int_{\hat{K}} {\bm u}\cdot {\nabla}\times {\varphi}\mathrm d S,\ \forall  {\varphi}\in  \tilde{Q}_{k-3}(  K)\right\},
	\end{align}
	where $\mathcal{D}=\nabla P_{k-5}( \hat K)\oplus [\bm {P}_{r-2} (\hat K)\slash \nabla P_{r-1}(\hat K)]$ when $k\geq 5$; $\mathcal{D}=[\bm {P}_{r-2} (\hat K)\slash \nabla P_{r-1}(\hat K)]$ when $k<5$ and $r\geq 2$; $\mathcal{D}=\emptyset$ when $k<5$ and $r<2$.
		\end{itemize}

\begin{figure}[!h]
\includegraphics[width=1\textwidth]{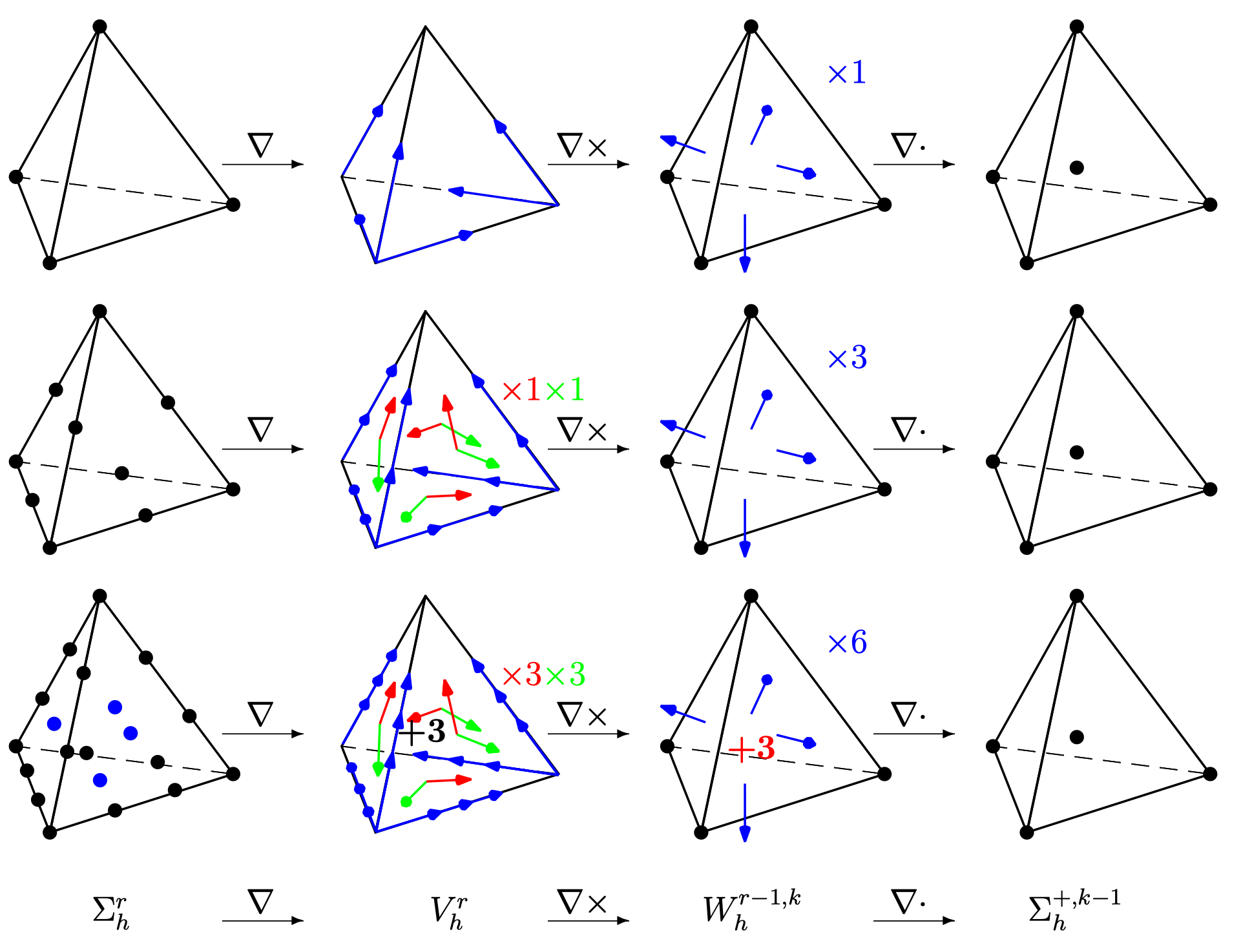}
\caption{The lowest-order ($k=2$) finite element complex \eqref{discrete-complex-r-k} on tetrahedra with $r=k-1$ in the first row, $r=k$ in the second row, and $r=k+1$ in the third row.}
\label{fig:third-family}
\end{figure}

\begin{lemma}\label{well-defined-conditions}
	The DOFs \eqref{tridef1-1}-\eqref{tridef1-4} are well-defined for any ${\bm u}\in \bm H^{1/2+\delta}({K})$ and ${\nabla}\cdot{\bm u}\in H^{3/2+\delta}({K})$ with $\delta>0$.
\end{lemma}
\begin{proof}
	It follows from the Cauchy-Schwarz inequality that the face DOFs \eqref{tridef1-3} and the interior DOFs \eqref{tridef1-4} are well-defined since  ${\bm u}\in \bm H^{1/2+\delta}({K})$ and ${\nabla}\cdot{\bm u}\in H^{3/2+\delta}({K})$. By the embedding theorem, we have ${\nabla}\cdot{\bm u}\in H^{3/2+\delta}({K})\hookrightarrow C^{0,\delta}(K)$, then the DOFs in  \eqref{tridef1-1} and \eqref{tridef1-2} are well-defined.
\end{proof}

\begin{lemma}
The DOFs for $W^{r-1, k}_{h}(K)$ are unisolvent. 
\end{lemma}
\begin{proof}
The decomposition \eqref{Vh2} is a direct sum. Therefore $\dim W^{r-1, k}_{h}(K) =\dim \nabla\times V^{r}_h(K)+\dim \Sigma^{+,k-1}_h(K)=\dim \nabla\times [R_r(K)\slash\nabla P_{r}(K)]+\dim P_{k-1}(K)=\dim  [R_r(K)\slash\nabla P_{r}(K)]+\dim P_{k-1}(K)=(r+2)(r+3)(2r-1)/6+k(k+1)(k+2)/6+1$ when $k\geq 5$ and $\dim W^{r-1, k}_{h}(K)=(r+2)(r+3)(2r-1)/6+k(k+1)(k+2)/6+2$ when $k=2,3,4$.  By counting the number of DOFs, we find the DOF set has the same dimension. Then it suffices to show that if all the DOFs vanish on a function $\bm{u}$, then $\bm{u}=0$. To see this, we first observe that $\nabla\cdot \bm{u}=0$ by the unisolvence of the DOFs of $\Sigma^{+,k-1}_h(K)$. Then $\bm{u}=\nabla\times\bm \phi\in \bm P_{r-1}(K)$ for some $\bm \phi\in V^{r}_h(K)$.  By the face DOFs of $W^{r-1, k}_{h}(K)$, 
$\bm u\cdot \bm {n}=0$ on faces.
By integration by parts, we have
\begin{align*}
(\bm u,\nabla q)_K=(\nabla\times\bm {\phi},\nabla  q)_K=\langle \nabla\times \bm \phi\cdot \bm n, q\rangle_{\partial K}=\langle  \bm u\cdot \bm n, q\rangle_{\partial K} =0 \text{ for any } q \in P_{r-1}(K),
\end{align*}	
which, together with the interior DOFs, leads to
\begin{align}\label{interior_0}
(\bm u,\bm q)_K=0 \text{ for any } \bm q\in \bm P_{r-2}(K).
\end{align}
Since $\bm u\cdot \bm n=0$ on faces, $\hat{\bm u}=(\hat x_1 \hat \varphi_1,\hat x_2 \hat  \varphi_2,\hat x_3 \hat  \varphi_3)^T$. Note that $\bm u$ and $\hat{\bm u}$ are related by \eqref{mapping-u}. 
Choosing $\bm{\varphi}= B_K^{-T}\hat{\bm \varphi} =B_K^{-T}(\hat \varphi_1,\hat \varphi_2,\hat  \varphi_3)^T\in \bm P_{r-2}(K)$, we have by \eqref{interior_0}:
$$0=(\bm u,\bm {\varphi})_K=\det(B_K)^{-1}(\hat{\bm u},\hat{\bm {\varphi}})_{\hat K}=\det(B_K)^{-1}\big((\hat x_1 \hat \varphi_1,\hat x_2 \hat  \varphi_2,\hat x_3 \hat  \varphi_3)^T,\hat{\bm \varphi}\big)_{\hat K}.$$
% we can choose $\phi$ as 
%\[\phi=x_1x_2(1-x_1-x_2)\psi \text{ for some $\psi \in P_{k-4}(K)$}.\]
%Note that we only write down for the reference case.
%Applying integration by parts for the vanishing interior DOFs, we obtain
%\[0=\left(\bm u,\bm q\right)=\left(\nabla \phi,\bm q\right)=-\left(\phi,\nabla\cdot\bm q\right)=\left(\lambda_{1}\lambda_{2}\lambda_{3}{\psi},{\psi}\right).\]
This implies that $\bm \varphi=0$ and hence $\bm u=0$.
\end{proof}

Provided $\bm u \in \bm H^{1/2+\delta}(\Omega)$, and $ \nabla \cdot \bm u \in H^{3/2+\delta}(\Omega)$ with $\delta >0$ (see Lemma \ref{well-defined-conditions}),  we can
define an $H(\grad \div)$ interpolation operator $\bm i_h$ whose restriction on $K$ is denoted as $\bm i_K$ and defined by
 \begin{align}\label{def-inte-01}
&\quad\bm M_v(\bm u-\bm i_K\bm u)=\{0\},\ \bm M_e(\bm u-\bm i_K\bm u)=\{0\},\\
 &\bm M_f(\bm u-\bm i_K\bm u)=\{0\},\text{ and } \bm M_K(\bm u-\bm i_K\bm u)=\{0\},
\end{align}
where $\bm M_v,\ \bm M_e, \ \bm M_f$ and $\bm M_K$ are the sets of DOFs in \eqref{tridef1-1}-\eqref{tridef1-4}.

Gluing the local spaces by the above DOFs, we obtain the global finite element spaces $\Sigma_h^{r}$, $V_h^{r}$, $W_h^{r-1,k}$, and $\Sigma_h^{+,k-1}$. 

% with the space function space $V_h(K)$ and the DOFs given above, we have constructed the space $V_h$. In light of the definitions of $\Sigma_h$, $V_h$, and $W_h$, the inclusions $\nabla\Sigma_h\subseteq V_h$ and $\nabla\times V_h\subseteq W_h$ hold. 
\begin{lemma} The following conformity holds:
$$
W^{r-1, k}_{h}\subset H(\grad\div; \Omega).
$$
\end{lemma}
\begin{proof}
It's straightforward since $\nabla\cdot W^{r-1, k}_{h}\subseteq \Sigma^{+,k-1}_h\subset H^{1}(\Omega)$.
\end{proof}

\subsection{Global finite element complexes for the quad-div problem}
With well defined global finite element spaces, we now develop some properties of the complex \eqref{discrete-complex-r-k} containing these spaces. 
\begin{theorem}
The complex \eqref{discrete-complex-r-k}  is  exact on contractible domains.
\end{theorem}
\begin{proof}
We first show the exactness at $V^{r}_h$ and $W^{r-1,k}_h$. To this end, we show that for any $\bm v_h\in V^{r}_h\subset H(\curl;\Omega)$ and $\bm u_h\in W^{r-1,k}_h\subset H(\grad \div;\Omega)\subset H(\div;\Omega)$ satisfying $\nabla\times\bm v_h=0$ and $\nabla\cdot\bm u_h=0$, there exists  $p_h\in \Sigma^{r}_h$ and $ \bm \phi_h\in \Sigma^{+,k-1}_h$ such that $\bm v_h=\nabla p_h$ and $\bm u_h=\nabla \times \bm \phi_h$.
Actually, this follows from the exactness of the standard finite element differential forms (e.g., \cite{arnold2018finite}).  To prove the exactness at $\Sigma^{+,k-1}_h$, that is to prove the operator $\nabla\cdot$ from $W^{r-1,k}_h$ to $\Sigma^{+,k-1}_h$ is surjective,  we  count the dimensions. The dimension count of the Lagrange elements reads:
 \[\dim \Sigma^{r}_h=\mathcal V+(r-1)\mathcal E+\frac{1}{2}(r-2)(r-1)\mathcal F+\frac{1}{6}(r-3)(r-2)(r-1)\mathcal K,\]
where $\mathcal V$, $\mathcal E$, $\mathcal F$, and $\mathcal K$ denote the number of vertices, edges, faces, and 3D cells, respectively.
%Moreover, $\dim \Sigma^{+,k-1}_h= \dim \Sigma^{k-1}_h$ for $k \geq 5$ and $\dim \Sigma^{+,k-1}_h= \dim \Sigma^{k-1}_h+\mathcal K$ for $k=2,3,4$.
The dimension count of the space $V_h^r$ reads:
\[\dim V_h^r=r\mathcal E+r(r-1)\mathcal F+\frac{1}{2}r(r-1)(r-2)\mathcal K.\]
From the DOFs \eqref{tridef1-1} -\eqref{tridef1-4}, 
\begin{align*}
	&\dim W^{r-1, k}_h-\dim \Sigma^{+,k-1}_h=\frac{1}{2}r(r+1)\mathcal F+\frac{1}{6}r(r+1)(2r-5)\mathcal K. 
	\end{align*}
From the above dimension count, we have 
\[-1+\dim \Sigma_h^r - \dim V^r_h + \dim W^{r-1, k}_h-\dim \Sigma^{+,k-1}_h=0,\]
where we have used Euler's formula $\mathcal V-\mathcal E+\mathcal F-\mathcal K=1$.
This completes the proof.
\end{proof}

We summarize the interpolations defined in Section \ref{sec:dofs} in the following diagram:
\begin{equation}\label{2complex}
\begin{tikzcd}
\mathbb{R} \arrow[r,"\subset"]  & H^1(\Omega) \arrow[d ]\arrow[r,"\nabla"]  & H(\curl;\Omega)\arrow[r,"\nabla\times"]\arrow[d]& H(\grad\div;\Omega)\arrow[r,"\nabla\cdot"]\arrow[d]  & H^1(\Omega)\arrow[r]\arrow[d ]& 0\\
\mathbb{R} \arrow[r,"\subset"]  & \Sigma \arrow[d, "\pi_h" ]\arrow[r,"\nabla"]  & V\arrow[r,"\nabla\times"]\arrow[d, "\bm r_h" ]  &W\arrow[r]\arrow[d, "\bm i_h" ]\arrow[r,"\nabla\cdot"]\arrow[d]  & \Sigma \arrow[r]\arrow[d ,"\tilde{\pi}_h"]& 0\\
\mathbb{R} \arrow[r,"\subset"]  & \Sigma^{r}_h\arrow[r,"\nabla"]  & V^{r}_h\arrow[r,"\nabla\times"]  & W^{r-1,k}_h\arrow[r,"\nabla\cdot"]  & \Sigma_h^{+,k-1} \arrow[r]& 0.\end{tikzcd}
\end{equation}
Here $\Sigma$, $V$, $W$ are three subspaces of $H^1(\Omega)$, $H(\curl;\Omega)$, and $H(\grad\div;\Omega)$ in which $\pi_h$ (or $\tilde\pi_h$), $\bm r_h$, and $\bm i_h$ are well-defined. 

Now we show that the interpolations in \eqref{2complex} commute with the differential operators. In addition to Lemma \ref{Vh}, this result also plays a key role in the error analysis below for the interpolations.
\begin{lemma}\label{commute}
The last two rows of the complex \eqref{2complex} are a commuting diagram, i.e.,
\begin{align}
	\nabla\pi_h u&=\bm r_h\nabla u \text{ for all } u\in \Sigma,\label{Pih_and_pih}\\
	\nabla\times\bm r_h \bm u&=\bm i_h\nabla\times\bm u \text{ for all } \bm u\in V,\label{Pih_and_tildepih}\\
	\nabla\cdot\bm i_h \bm u&=\tilde \pi_h\nabla\cdot\bm u \text{ for all } \bm u\in W. \label{wh_and_tildepih}
\end{align}
\end{lemma}
\begin{proof} The proof of \eqref{Pih_and_pih} can be found in \cite[Chapter 5]{Monk2003}. A similar  trick can be used to prove \eqref{Pih_and_tildepih} and \eqref{wh_and_tildepih}. For simplicity of presentation, we omit it. %we only show the proof of \eqref{Pih_and_tildepih}.
	\end{proof}

The following lemma relates the interpolation on $K$ to that on $\hat K$. 

\begin{lemma}\label{relation}
For $\bm u\in W$, under the transformation \eqref{mapping-u}, we have
$\widehat{\bm i_K \bm u}=\bm i_{\hat K}\hat {\bm u}$.
\end{lemma}
\begin{proof}
By the transformations \eqref{mapping-u}, \eqref{mapping-curlu}, and \eqref{n}, we have all the DOFs in \eqref{tridef1-4} and the second set of \eqref{tridef1-3} are equivalent with those for $\hat{\bm u}$ on $\hat K$. In addition, all the  DOFs in \eqref{tridef1-1}, \eqref{tridef1-2}, and the first set of \eqref{tridef1-3} are differed from those for $\hat{\bm u}$ on $\hat K$ by a factor $\frac{1}{\det(B_K)}$. 
This means the DOFs for defining $\widehat{\bm i_K \bm u}$ and those for defining $\bm i_{\hat K}\hat {\bm u}$
are differed by a constant factor.
According to Proposition 3.4.7 in \cite{brenner2008mathematical}, we complete the proof.
\end{proof}

Next, we establish the approximation property of the interpolation operator.

\begin{theorem}
For $r=k-1$, $k$, or $k+1$, if $\bm u\in \bm H^{s+(r-k)}(\Omega)$ and $\nabla\cdot\bm u\in  H^s(\Omega)$, $3/2+\delta\leq s \leq k$ with $\delta>0$, then we have the following error estimates for the interpolation $\bm i_h$,
	\begin{align}
	&\left\|\bm u-\bm i_h\bm u\right\|\leq Ch^{s+(r-k)}(\left\|\bm u\right\|_{s+(r-k)}+\left\|\nabla\cdot\bm u\right\|_{s}),\label{inter-u}\\
	&\left\|\nabla\cdot(\bm u-\bm i_h\bm u)\right\|\leq Ch^s\left\|\nabla\cdot\bm u\right\|_{s},\label{inter-curlu}\\
	&	\left|\nabla\cdot(\bm u-\bm i_h\bm u)\right|_1\leq Ch^{s-1}\left\|\nabla\cdot\bm u\right\|_{s}.\label{inter-divdivu}
	\end{align}
\end{theorem}
\begin{proof}
We only prove the results for integer $s$ to avoid the technical complications.  To prove \eqref{inter-u},  we first apply the transformation \eqref{mapping-u} and Lemma \ref{relation} to derive, for a general element $K\in \mathcal{T}_h$,
		\begin{align*}
		\left\|\bm u-\bm i_K\bm u\right\|_K&=\left(\int_{\hat{K}}\det(B_K)^{-1}\left| B_K(\hat{\bm u}-\widehat{\bm i_K\bm u})\right|^2\d \hat V\right)^{\frac{1}{2}}\\
		\leq &\left|\det(B_K)\right|^{-\frac{1}{2}}\left\|B_K\right\|
		\left\|\hat{\bm u}-{\bm i_{\hat K}\hat{\bm u}}\right\|_{\hat K}.
		\end{align*}
		From Lemma \ref{Vh}, $\bm P_{r-1}(K)\subseteq W^{r-1,k}_h(K)$.
		Therefore, we have $\bm i_{\hat K}\hat {\bm p}=\hat{\bm p}$ when $\hat{\bm p}\in \bm P_{r-1}(\hat K)$.  We obtain, with the help of Lemma \ref{well-defined-conditions} and Theorem 5.5 in \cite{Monk2003},
		\begin{align*}
		&\quad \left\|\hat{\bm u}-{\bm i_{\hat K}\hat{\bm u}}\right\|_{\hat K}
		=\left\|\left(I-\bm i_{\hat K}\right)(\hat{\bm u}+\hat{\bm p})\right\|_{\hat K}\\
		\leq  &\inf_{\hat {\bm p}\in \bm P_{r-1}(\hat K)}C\left(\left\|\hat{\bm u}+\hat{\bm p}\right\|_{s+(r-k),\hat K}+\|\hat{\nabla}\cdot(\hat{\bm u}+\hat{\bm p})\|_{s,\hat K}\right)\\
		 &\quad\quad \leq C\left(\left|\hat{\bm u}\right|_{s+(r-k),\hat K}+|\hat{\nabla}\cdot\hat{\bm u}|_{s,\hat K}\right).
		\end{align*}
Mapping back to the general element $K$ leads to
		%       and summing over $K\in\mathcal{T}_h$ leads to
\begin{align}\label{uestimate_on_K}
	\left\|\bm u-\bm i_K\bm u\right\|_K\leq C{h^{s+(r-k)}}(\left\|\bm u \right\|_{s+(r-k),K}+\left\| \nabla\cdot \bm u\right\|_{s,K}).
\end{align}
Summing the above inequality	 \eqref{uestimate_on_K} over 
$K\in\mathcal{T}_h$, we obtain \eqref{inter-u}.

As for \eqref{inter-curlu} and \eqref{inter-divdivu}, we use Lemma \ref{commute} and Lemma \ref{Vh} to obtain, for $i=0,1$,
\begin{align*}
	\left\|\nabla\cdot(\bm u-\bm i_h\bm u)\right\|_i=\left\|\nabla\cdot \bm u-\bm \tilde \pi_h\nabla\cdot\bm u)\right\|_i\leq Ch^{s-i}\left\|\nabla\cdot\bm u\right\|_{s},
\end{align*}
where we have used the standard approximation property of the Lagrange finite element.
\end{proof}

\section{Three families of grad-div conforming elements on cuboids}

In this section, we will construct three families of grad-div conforming finite element spaces $W^{r-1,k}$ for a cuboid mesh by taking $r=k-1$, $k$, and $k+1$ in  \eqref{discrete-complex-r-k}.

Proceeding as in the tetrahedral case, we choose the Lagrange element space of order $r$ on cuboids for $\Sigma_h^r$. The space  $\Sigma_{h}^{+,k-1}$ is the $(k-1)$-th order Lagrange element space, which is enriched with an interior bubble function when $k=2$. In addition, we choose the first N\'ed\'elec element space of order $r$ for $V_h^r$. The DOFs for these spaces can be found in \cite[Chapter 6]{Monk2003}. In the following, we will define DOFs for $W^{r-1,k}$. The three grad-div elements with $r=k-1$, $r=k$, and $r=k+1$ with $k=2$ for a cuboid element are shown in Fig. \eqref{fig:third-family-cube}

\begin{itemize}
		\item Vertex DOFs $\bm M_{ {v}}({\bm u})$ at all vertices $ {v}_{i}$ of $K$:
	\begin{equation*}
	\bm M_{ {v}}({\bm u})=\left\{( \nabla\cdot {\bm u})(  v_{i}),\; i=1,\;2,\cdots,8\right\}.
	\end{equation*}
	\item Edge DOFs $\bm M_{ {e}}(  {\bm u})$ at all edges $ {e}_i$ of $ {K}$:
	\begin{align*}
		 \bm M_{ {e}}(  {\bm u})=	\left\{\frac{1}{\operatorname{length}(e_i)}\int_{e_i}\nabla\cdot{\bm u}q\d s,\ \forall  {q}\in P_{k-3}( {e}_i), i=1,2,\cdots,12\right\}.
	\end{align*}	
	\item Face DOFs $\bm M_{f}({\bm u})$ at all faces $ {f}_i$ of ${K}$ (with the unit  normal vector $ {\bm n}_i$):
	\begin{align*}
		 \bm M_{f}({\bm u})=&\left\{\frac{1}{\operatorname{area}(f_i)}\int_{f_i}\nabla\cdot{\bm u}q\d A,\ \forall  {q}\in Q_{k-3,k-3}( {f}_i), i=1,2,\cdots,6\right\}\nonumber\\
		 &\cup \left\{\int_{f_i}{\bm u}\cdot \bm n_iq\d A,\ \forall  {q}\in Q_{r-1,r-1}( {f}_i), i=1,2,\cdots,6\right\},
	\end{align*}	
    \item Interior DOFs $\bm M_{ {K}}(  {\bm u})$:
	\begin{align*}
	&\bm M_{ {K}}(  {\bm u})=\left\{\int_{ {K}}  {\bm u}\cdot  {\bm q}\mathrm \d V,\ \forall \bm q=B_K^{-T}\hat{ \bm q},\ \hat{\bm q} \in  \mathcal{D} \right\},
	%\label{def5}
	%&\bm M_{ {K}}(  {\bm u})=\left\{\int_{\hat{K}} {\bm u}\cdot {\nabla}\times {\varphi}\mathrm d S,\ \forall  {\varphi}\in  \tilde{Q}_{k-3}(  K)\right\},
	\end{align*}
	where $\mathcal{D}=\nabla Q_{k-3}( \hat K)\oplus [Q_{r-2,r-1,r-1}(\hat K)\times Q_{r-1,r-2,r-1}(\hat K)\times Q_{r-1,r-1,r-2}(\hat K)\slash \nabla Q_{r-1}(\hat K)]$ when $k\geq 3$; $\mathcal{D}=[Q_{r-2,r-1,r-1}(\hat K)\times Q_{r-1,r-2,r-1}(\hat K)\times Q_{r-1,r-1,r-2}(\hat K)\slash \nabla Q_{r-1}(\hat K)]$ when $k=2$ and $r\geq 2$; $\mathcal{D}=\emptyset$ when $k=2$ and $r<2$.
		\end{itemize}

The same theoretical results as the tetrahedral elements can be obtained by a similar argument. We omit them for this case.

\begin{figure}[!t]
\includegraphics[width=1\textwidth]{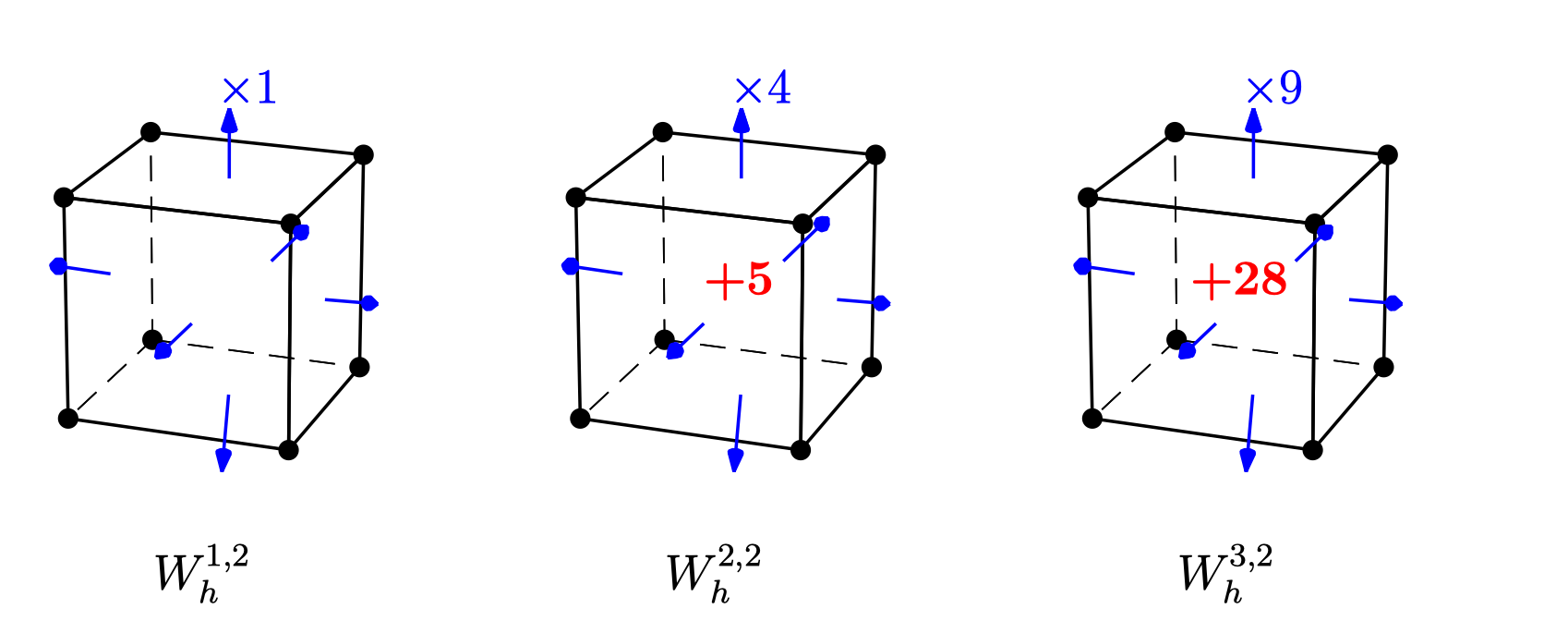}
\caption{The three versions of lowest-order ($k=2$) grad-div finite elements on cuboids}
\label{fig:third-family-cube}
\end{figure}

\section{The application to quad-div problems}
In this section, we use the $H(\grad \div)$-conforming finite elements  to solve the quad-div problem which is stated as follows:

For $\bm  f\in H(\curl;\Omega)$ and $\bm g\in L^2(\Omega)$, find $\bm u$, such that
\begin{equation}\label{prob1}
\begin{split}
(\nabla\nabla\cdot)^2\bm u+\bm u&=\bm f\ \ \text{in}\;\Omega,\\
\nabla \times \bm u &=\bm g\ \ \text{in}\;\Omega,\\
\bm u\cdot\bm n&=0\ \ \text{on}\;\partial \Omega,\\
\nabla \cdot  \bm u&=0\ \  \text{on}\;\partial \Omega.
\end{split}
\end{equation}
%Here $H(\curl^0;\Omega)$ is the space of $\bm L^2(\Omega)$ functions with vanishing curl, i.e., 
%\[H(\text{curl}^0;\Omega) :=\{\bm u\in {\bm L}^2(\Omega):\; \nabla\times \bm u=0\},\] and  $\bm n$ is the unit outward normal vector to $\partial \Omega$.
Here, to make the problem consistent, $\bm g=\nabla\times \bm f$.
By taking curl on both sides of the first equation of \eqref{prob1}, we see that the condition $\nabla\times\bm u=\bm g$ holds automatically.

We define $H_0(\grad\div;\Omega)$ and $H_0(\div;\Omega)$ with vanishing boundary conditions:
\begin{align*}
H_0(\grad \div;\Omega)&:=\{\bm u \in H(\grad \div;\Omega):\;\bm u\cdot{\bm n}=0\; \text{and}\; \nabla\cdot \bm u=0\;\; \text{on}\ \partial \Omega\},\\
H_0(\div;\Omega)&:=\{\bm u \in H(\div;\Omega):\;\bm u\cdot{\bm n}=0\; \text{on}\ \partial \Omega\}.
\end{align*}
%The space of $\bm L^2(D)$ functions with a divergence zero is denoted by $H(\div^0;D)$ and defined by
%\[H(\text{div}^0;D) :=\{\bm u\in {\bm L}^2(D):\; \nabla\cdot \bm u=0\}.\]
%For the sake of satisfying divergence-free condition, we adopt mixed methods where the constraint $\nabla\cdot\bm u=0$ in \eqref{prob1} is satisfied in a weak sense by introducing an auxiliary unknown $p$ and employing a mixed
The variational formulation is to seek $\bm u\in H_0(\grad\div;\Omega)$  such that
\begin{equation}\label{prob22}
\begin{split}
a(\bm u,\bm v)&=(\bm f, \bm v)\quad \forall \bm v\in H_0(\grad\div;\Omega),
\end{split}
\end{equation}
with $a(\bm u,\bm v):=\big(\nabla(\nabla\cdot\bm u),\nabla(\nabla\cdot\bm v)\big) + (\bm u,\bm v)$.

It follows from Lax-Milgram Lemma that \eqref{prob22} is well-posedness.
Taking $\bm v=\nabla\times\bm \psi$ with $\bm \psi \in H_0(\curl;\Omega)$ in \eqref{prob22} leads to
\[(\bm u, \nabla\times\bm \psi)=(\bm f, \nabla\times\bm \psi)=(\bm g, \bm \psi),\]
which implies $\nabla\times \bm u=\bm g$ holds in the sense of $\bm L^2(\Omega)$.
Since the regularity of the solution plays a crucial role in the error analysis, we will first derive a regularity result for the quad-div problem before proceeding further.

\begin{lemma}\label{reglaplace}
	Suppose $\Omega$ is a polyhedron. Consider the problem $-\Delta u=f$ in $\Omega$ with $u=0$ on $\partial \Omega.$ There exists a constant $\alpha_0>1/2$ satisfying the same conditions as in \cite[Theorem 2.2.1]{reg-laplace} such that
	\[u\in H^{1+\alpha_0}(\Omega) \text{ when } f\in H^{\alpha_0-1}(\Omega).\]
In particular, when $\Omega$ is convex, $\alpha_0$ is at least 1.
\end{lemma}
\begin{theorem}\label{regularity}
	Under the assumption of $\Omega$, we assume further $\Omega$ is polyhedron. For $1/2<\alpha\leq\min\{\alpha_0,1\}$ with $\alpha_0$ defined in Lemma \ref{reglaplace}, the solution $\bm u$ of \eqref{prob1} satisfies
	$\bm u\in \bm  H^{\alpha}(\Omega)$ and $\nabla \cdot \bm u\in H^{1+\alpha}(\Omega)$. Moreover, it admits the following decomposition:
\begin{align*}
	\bm u=\nabla\times \bm w^{sing}+\bm u^{reg}
\end{align*}
with $\bm u^{reg}\in \bm H^{2+\alpha}(\Omega)$ and $\bm w^{sing}\in \bm H^{1+\alpha}(\Omega)$.
\end{theorem}
\begin{proof}
	Let $w=\nabla\cdot \bm u$, then, from the first equation of \eqref{prob1}, we have 
	\[\nabla(\nabla\cdot\nabla w)=\nabla(\Delta w)=\bm f-\bm u\in \bm L^2(\Omega),\]
	which implies $\Delta w\in H^1(\Omega).$
	Note that $w=\nabla\cdot \bm u=0$ on $\partial \Omega$. Applying Lemma \ref{reglaplace},
	we obtain $w=\nabla\cdot\bm u\in H^{1+\alpha_1}(\Omega)$ with $1/2<\alpha_1\leq\alpha_0.$
	
	Since $\bm u\in H_0(\div;\Omega)\cap H(\curl;\Omega)\hookrightarrow \bm H^{\alpha_2}(\Omega)$ with $1/2<\alpha_2\leq \min\{1,\alpha_0\}$ \cite[Proposition 3.7]{amrouche1998vector}, we have $\bm u\in \bm H^{\alpha_2}(\Omega)$.
Taking $\alpha=\min\{\alpha_1,\alpha_2\}$ yields $\bm u\in \bm  H^{\alpha}(\Omega)$ and $\nabla \cdot \bm u\in H^{1+\alpha}(\Omega)$. 

	Let $\mathcal{O}$ be a bounded, smooth, contractible open set with $\bar{\Omega}\subset\mathcal{O}$.
For the solution $\bm u$, we can extend $\bm u$ in the following way:
\begin{align*}
\bm {\tilde{u}}&=\begin{cases}
\bm u,&\Omega,\\
0,&\mathcal{O}-\bar{\Omega}.\\
\end{cases}
\end{align*}
It follows from $\bm u\in H_0(\div;\Omega)$ and $\nabla\cdot \bm u\in H^{1+\alpha}(\Omega)\cap H_0^1(\Omega)$ that $\bm {\tilde{u}}\in H_0(\div;\mathcal{O})$ and $\nabla \cdot \bm {\tilde{u}}\in H_0^1(\mathcal{O})\cap H^{1+\alpha}(\mathcal O)$. 
We consider the problem of
finding $\psi$ defined in $\mathcal{O}$ such that
\begin{align}\label{lap-01}
-\triangle \psi&=-\nabla \cdot \bm {\tilde{u}}\in H^{1+\alpha}(\mathcal O)\ \text{in}\ \mathcal{O},\\
\psi&=0\ \text{on}\ \partial \mathcal{O}.\label{lap-02}
\end{align}
By the regularity result of the Laplace problem \cite[Theorem 1.8]{Girault2012Finite}, there exists a function $\psi\in H^{3+\alpha}(\mathcal{O})$ satisfying \eqref{lap-01} and \eqref{lap-02}.
%\begin{align}\label{reg1}
%	\left\|\psi\right\|_{3+\alpha_1,\mathcal O}\leq C\|\nabla \cdot \bm {\tilde{u}}\|_{1+\alpha_1,\mathcal O}.
%\end{align}
Rewriting \eqref{lap-01} and restricting on $\Omega$, we have
$$
\nabla\cdot(\bm {{u}}-\nabla \psi)=0 \text{ in }  \Omega
$$
with $\bm {{u}}-\nabla \psi\in \bm H^{\alpha}(\Omega)$. According to \cite[Remark 3.12]{Girault2012Finite}, there exists $\bm w \in \bm H^{1+\alpha}(\Omega)$ such that
\begin{align}\label{curlw}
	\bm {{u}}-\nabla \psi=\nabla\times\bm w \text { and } \nabla\cdot \bm w=0 \text { in } \Omega.
\end{align}
Denoting $\bm u^{reg}=\nabla\psi\in \bm H^{2+\alpha}(\Omega)$ and $\bm w^{sing}=\bm w\in \bm H^{1+\alpha}(\Omega)$, we have, from \eqref{curlw}, that
$\bm u=\nabla\times \bm w^{sing}+\bm u^{reg}$.
\end{proof}

%\begin{remark}
%	As a special case, when $g=0$, it follows from \cite[Theorem 2.9]{Girault2012Finite} that $\bm u=\nabla \varphi$ for $\varphi\in H^1(\Omega)$. The problem \eqref{prob1} degenerates to a biharmonic problem of finding $\varphi$ such that
%	 \begin{align*}
%	 	&\Delta^2\varphi=f\in H^1(\Omega) \text{ in } \Omega,\\
%	 	&\nabla\varphi\cdot\bm n=0 \text{ on } \partial\Omega,\\
%	 	&\Delta\varphi=0 \text{ on } \partial\Omega,
%	 \end{align*}
%	 which can be decomposed into two Laplace problems. The regularity result of the Laplace problem \cite{Girault2012Finite} shows that $\bm u\in \bm H^2(\Omega)$ when $\partial\Omega$ is $C^{2,1}$,
%	where $C^{2,1}$ is the standard notation of boundary smoothness.
%\end{remark}

\begin{remark}
	With Theorem \ref{regularity}, the interpolation $\bm i_h\bm u$ is well-defind.
\end{remark}

 We now present the finite element scheme. We define the finite element space with vanishing boundary conditions
\begin{eqnarray*}%classical Raviart-Thomas-N\'ed\'elec's
  W^0_h=\{\bm{v}_h\in W^{r-1,k}_h,\ \bm{n} \cdot \bm{v}_h=0\ \text{and}\ \nabla\cdot \bm{v}_h = 0 \ \text {on} \ \partial\Omega\}.
\end{eqnarray*}
The $H(\grad \div)$-conforming finite element method reads: seek $\bm u_h\in W^0_h$,  such that
\begin{equation}\label{prob3}
\begin{split}
 a(\bm u_h,\bm v_h)&=(\bm f, \bm v_h)\quad \forall \bm v_h\in W^0_h.
\end{split}
\end{equation}

The following approximation property of $\bm u_h$ follows immediately from C\'ea's lemma and the duality argument. 
	\begin{theorem}\label{uh_approx}
For $r=k-1$, $r=k$, or $r=k+1$, if $\bm u\in \bm H^{s+(r-k)}(\Omega)$ and $\nabla\cdot\bm u\in  H^s(\Omega)$, $3/2+\delta\leq s \leq k$ with $\delta>0$, then we have the following error estimates for the numerical solution $\bm u_h$,
	\begin{align}
	\left\|\bm u-\bm u_h\right\|&\leq Ch^{\min\{s+(r-k),2(s-1)\}}(\left\|\bm u\right\|_{s+(r-k)}+\left\|\nabla\cdot\bm u\right\|_{s}),\label{uh}\\
	&\left\|\nabla\cdot(\bm u-\bm u_h)\right\|\leq Ch^{\min\{s,2(s-1)\}}\left(\left\|\bm u\right\|_{s-1}+\left\|\nabla\cdot\bm u\right\|_{s}\right),\label{divuh}\\
	&\left|\nabla\cdot(\bm u-\bm u_h)\right|_1\leq Ch^{s-1}\left(\left\|\bm u\right\|_{s-1}+\left\|\nabla\cdot\bm u\right\|_{s}\right).\label{divdivuh}
	\end{align}
\end{theorem}
\begin{remark}
	When $\Omega$ is a Lipschitz polyhedron, even with the lowest regularity, the scheme still has a convergence order $1/2$ in $H(\grad\div)$ norm. 
\end{remark}

\section{Numerical Experiments}
We now turn to a concrete example to test our new elements. We consider the problem \eqref{prob1} on a unit cube $\Omega=(0,1)\times(0,1)\times(0,1)$ with an exact solution
	\begin{equation}
	\bm u=\nabla \left(x^3y^3z^3(x-1)^3(y-1)^3(z-1)^3\right).
	\end{equation}
   Then, by a simple calculation, the source term $\bm f$ can be derived. Note that in this case $g=0$. We denote the finite element solution  as $\bm u_h$. To measure the error between the exact solution and the finite element solution, we also denote 
	\[\bm e_h=\bm u-\bm u_h.\]

%\subsection{The grad-div conforming finite elements on tetrahedra}
\begin{example}
In this example, we test the tetrahedral elements. To this end, we partition the unit cube into $N^3$ small cubes and then partition each small cube into 6 congruent tetrahedra. We  use the lowest-order elements in  three families  to solve the problem \eqref{prob1} on the uniform tetrahedral mesh. 

Tables \ref{tetra-tab1}, \ref{tetra-tab2}, and \ref{tetra-tab3} illustrate various errors and convergence rates for three families.  
We observe from the tables that the numerical solution converges to the exact solution with a convergence order 1 for the family $r=k-1$, 2 for the family $r=k$, and $2$ for the family $r=k+1$ in the sense of $L^2$-norm. In addition,  the three families have the same convergence order 2 in the $H(\div)$-norm and 1 in the $H(\grad \div)$-norm, respectively. All the results coincide with Theorem \ref{uh_approx}, which confirms the correctness of the elements and their properties.

\begin{table}[!ht]
	\centering
	\caption{Example 1: Numerical results by  the lowest-order $(k=2)$ tetrahedral element in the  family $(r=k-1)$ of $H(\grad \div)$-conforming elements} \label{tetra-tab1}
	\begin{tabular}{cccccccc}
		\hline
		$N$ &$\left\|\bm e_h\right\|$&rates&$\left\|\nabla\cdot\bm e_h\right\|$&rates&$\left\|\nabla(\nabla\cdot\bm e_h)\right\|$& rates\\
		\hline
		$ 16$ &7.338806e-07  && 3.773907e-06& &1.261805e-04		& \\
		$ 20$&5.585337e-07 &1.2236&2.462834e-06&1.9127&1.016297e-04 &0.9697\\
		$24$&4.511530e-07& 1.1711&1.728736e-06&1.9412&8.500500e-05& 0.9797& \\
		$28$&3.788654e-07& 1.1328&1.278389e-06&1.9578&7.302452e-05&0.9855& \\
        $32$&3.268841e-07& 1.1052&9.829309e-07&1.9682&6.398944e-05&0.9891& \\
		\hline
	\end{tabular}
\end{table}

\begin{table}[!ht]
	\centering
	\caption{Example 1: Numerical results by  the lowest-order $(k=2)$ tetrahedral element in the  family $(r=k)$ of $H(\grad \div)$-conforming elements} \label{tetra-tab2}
	\begin{tabular}{cccccccc}
		\hline
		$N$ &$\left\|\bm e_h\right\|$&rates&$\left\|\nabla\cdot\bm e_h\right\|$&rates&$\left\|\nabla(\nabla\cdot\bm e_h)\right\|$& rates\\
		\hline
		$ 8$ &1.232033e-06 && 1.150197e-05& &3.902786e-04 & \\
		$ 12$ &5.905553e-07&1.8136& 5.614381e-06&1.7688 &1.654137e-04& 0.8952 \\
		$ 16$&3.416300e-07 &1.9026&3.269987e-06&1.8790 &1.259942e-04& 0.9462\\
		$ 20$&2.215699e-07&1.9404&2.127621e-06&1.9260&1.015312e-04&0.9674\\
		$ 24$&1.549977e-07&1.9599&1.490982e-06&1.9502&8.494707e-05&0.9782\\
		\hline
	\end{tabular}
\end{table}

\begin{table}[!ht]
	\centering
	\caption{Example 1: Numerical results by  the lowest-order $(k=2)$ tetrahedral element in the  family $(r=k+1)$ of $H(\grad \div)$-conforming elements} \label{tetra-tab3}
	\begin{tabular}{cccccccc}
		\hline
		$N$ &$\left\|\bm e_h\right\|$&rates&$\left\|\nabla\cdot\bm e_h\right\|$&rates&$\left\|\nabla(\nabla\cdot\bm e_h)\right\|$& rates\\
		\hline
		$ 8$ &1.224295e-06&& 1.149723e-05& &2.377994e-04& \\
		$ 10$&8.220074e-07& 1.7853  &7.812974e-06&1.7313& 1.954954e-04&0.8779\\
		$ 12$ &5.864916e-07  & 1.8516& 5.613355e-06& 1.8135 &1.654135e-04&0.9165 \\
		$ 14$&4.381462e-07& 1.8917  &4.211742e-06 &1.8636& 1.431136e-04 &0.9394\\
		$ 16$&3.391664e-07 &1.9176 &3.269652e-06&1.8961&1.259941e-04 &0.9541\\
		\hline
	\end{tabular}
\end{table}

\end{example}

\begin{example}

In this example, we test the cuboid grad-div conforming elements. We use uniform cubic meshes with the mesh size $h$ varying from ${1}/{12}$ to ${1}/{20}$. 
Unlike tetrahedral elements, in this test, we also explore superconvergence of the cuboid elements. To this end, we denote $\{w_n\}_{n=1}^p$ and $\{g_n\}_{n=1}^p$ as the weights and nodes of Legendre-Gauss quadrature rule of an order $p$. We also denote $\{w_n^l\}_{n=1}^p$ and $\{l_n\}_{n=1}^p$ as the weights and nodes of  Legendre-Gauss-Lobbato quadrature rule of an order $p$. 
For $\bm u=(u_1,u_2,u_3)^T$, we define three discrete norms. 
\begin{align*}
&\3bar\bm u\3bar_{U}^2 = \sum_{K\in \mathcal{T}_h} \sum_{r,s,t=1}^k\omega_r^l\omega_s^l\omega_t^l\Big(h_x^Kh_y^Kh_z^K|\bm u(x_c^K+h_x^Kl_r,y_c^K+h_y^Kl_s,z_c^K+h_z^Kl_t)|^2\Big),\\
&\3bar\bm u\3bar_{V}^2\!=\!\sum_{K\in \mathcal{T}_h}\!\sum_{m,n=1}^{k+r-2}\omega_m\omega_n\!\Big(\!h_y^Kh_z^K
\left\|u_1(\cdot,y_c^K+h_y^Kg_m,z_c^K+h_z^Kg_n)\right\|^2\!+\!h_x^K
h_z^K\\
&~~\left\|u_2(x_c^K\!+\!h_x^Kg_m,\cdot,z_c^K\!+\!h_z^Kg_n)\right\|^2
\!+\!h_x^Kh_y^K\left\|u_3(x_c^K\!+\!h_x^Kg_m,y_c^K
\!+\!h_y^Kg_n,\cdot)\right\|^2\!\Big),
\end{align*}
and
\begin{align*}
\3bar\bm u\3bar_{W}^2=\sum_{K\in \mathcal{T}_h}\sum_{n=1}^{k-1}\omega_l\Big(&h_x^K\left\|u_1(x_c^K+h_x^Kg_n,\cdot,\cdot)\right\|^2
+h_y^K\left\|u_2(\cdot,y_c^K+h_y^Kg_n,\cdot)\right\|^2\\
&+h_z^K\left\|u_3(\cdot,\cdot, z_c^K+h_z^Kg_n)\right\|^2\Big),
\end{align*}
where $(x_c^K,y_c^K,z_c^K)$ is the center of element $K$ and  $2h_x^K,2h_y^K,2h_z^K$ are the lengths of edges parallel to $x,y,z$ axes, respectively.
%\begin{figure}[!h]
%\includegraphics[width=7.5cm]{rec.jpg}
%\caption{Error curves in different norms}\label{Fig:cube}
%\end{figure}
 
Tables \ref{cube_tab1}, \ref{cube_tab2}, and \ref{cube_tab3} shows errors measured in various norms for the lowest-order cuboid elements in the three familes. We also depict error curves with a log-log scale in Fig. \ref{Fig:errorcurve}.
From Fig. \ref{Fig:errorcurve} (A), we can observe superconvergence phenomena that $\3bar\nabla(\nabla\cdot\bm e_h)\3bar_W$ and $\3bar\bm e_h\3bar_V$ converge to 0 with one order higher than $\left\|\nabla(\nabla\cdot\bm e_h)\right\|$ and $\left\|\bm e_h\right\|$. In addition, from Fig. \ref{Fig:errorcurve} (B)(C), we  can observe superconvergence of $\3bar\nabla(\nabla\cdot\bm e_h)\3bar_W$. 

We can not observe any superconvergence of $\3bar\bm e_h\3bar_V$ for $r=2,3$ and $\3bar\nabla\cdot\bm e_h\3bar_U$ for all the 3 families when $k=2$. To further investigate the superconvergence of $\nabla\cdot\bm e_h$, we test the third-order $(k=3)$ element in the family of $r=k-1$.  The results are shown in Table \ref{cube_tab4} and Fig. \ref{Fig:errorcurve}(D). In this case, we can observe superconvergence of $\nabla\cdot\bm e_h$.

Using these superconvergent results, together with some recovery techniques, we can construct a solution with higher accuracy if needed, which is one of the reasons that we explore the superconvergence of cuboid elements. 

We conclude this section by pointing out that the three families of elements bear their own advantages. 
The  family of $r=k-1$ can be the best choice if we pursue a low computational cost, while the  family with $r=k+1$ stands out for its higher accuracy in the $L^2$-norm without any recovery techniques.

\begin{table}[!ht]
	\centering
	\caption{Example 2: Numerical results by the lowest-order $(k=2)$ cuboid element in the family $(r=k-1)$ of $H(\grad \div)$-conforming elements} \label{cube_tab1}
	\begin{tabular}{cccccccc}
		\hline
		$h$ &$\left\|\bm e_h\right\|$&$\left\|\bm e_h\right\|_V$&$\left\|\nabla\cdot\bm e_h\right\|$&$\left\|\nabla\cdot\bm e_h\right\|_U$&$\left\|\nabla(\nabla\cdot\bm e_h)\right\|$& $\left\|(\nabla\times)^2\bm e_h\right\|_W$\\
		\hline
		$1\slash 8$ &1.2939e-06&8.2349e-07&6.6566e-06  & 2.7601e-06 &1.5795e-04
		&6.2427e-05\\
		$1\slash16$&5.6099e-07 &2.1371e-07&1.7020e-06& 6.6734e-07&7.6700e-05& 1.5975e-05\\
		$1\slash24$&3.6063e-07  &9.5663e-08&7.5957e-07 & 2.9471e-07 & 5.0814e-05&7.1306e-06\\
		$1\slash32$&2.6677e-07&5.3946e-08 &4.2787e-07&1.6541e-07 &3.8025e-05 &4.0170e-06  \\
		$1\slash40$&2.1201e-07&3.4566e-08
 &2.7402e-07&1.0575e-07&3.0388e-05&2.5726e-06 \\
		\hline
	\end{tabular}
\end{table}

\begin{table}[!ht]
	\centering
	\caption{Example 2: Numerical results by the lowest-order $(k=2)$ cuboid element in the  family $(r=k)$ of $H(\grad \div)$-conforming elements} \label{cube_tab2}
	\begin{tabular}{cccccccc}
		\hline
		$h$ &$\left\|\bm e_h\right\|$&$\left\|\bm e_h\right\|_V$&$\left\|\nabla\cdot\bm e_h\right\|$&$\left\|\nabla\cdot\bm e_h\right\|_U$&$\left\|\nabla(\nabla\cdot\bm e_h)\right\|$& $\left\|(\nabla\times)^2\bm e_h\right\|_W$\\
		\hline
		$1\slash 4$ &2.2275e-06 &2.1791e-06&  2.0877e-05&1.4226e-05&3.2323e-04
		&2.0116e-04  \\
		$1\slash 10$&3.2909e-07 &3.2124e-07& 3.2354e-06 &2.4023e-06&1.2317e-04 &3.1825e-05 \\
		$1\slash16$&1.2730e-07 & 1.2419e-07  &1.2547e-06 &9.2031e-07 &7.6282e-05 & 1.2340e-05  \\
		$1\slash22$&6.7137e-08& 6.5485e-08&6.6217e-07 & 4.8348e-07 & 5.5322e-05 & 6.5115e-06\\
		$1\slash28$&4.1395e-08&4.0373e-08 & 4.0839e-07&2.9756e-07&4.3414e-05&4.0157e-06\\
		\hline
	\end{tabular}
\end{table}

 \begin{table}[!ht]
	\centering
	\caption{Example 2: Numerical results by the lowest-order $(k=2)$ cuboid element in the  family $(r=k+1)$ of $H(\grad \div))$-conforming elements} \label{cube_tab3}
	\begin{tabular}{cccccccc}
		\hline
		$h$ &$\left\|\bm e_h\right\|$&$\left\|\bm e_h\right\|_V$&$\left\|\nabla\cdot\bm e_h\right\|$&$\left\|\nabla\cdot\bm e_h\right\|_U$&$\left\|\nabla(\nabla\cdot\bm e_h)\right\|$& $\left\|(\nabla\times)^2\bm e_h\right\|_W$\\
		\hline
		$1\slash 4$ &2.5839e-06&2.5818e-06 &  2.0796e-05&1.4263e-05& 3.2318e-04
		&2.0132e-04 \\
		$1\slash 10$&3.2119e-07&3.2115e-07&3.2315e-06 &2.4030e-06& 1.2317e-04 &3.1829e-05\\
		$1\slash16$&1.2417e-07 & 1.2416e-07 &1.2541e-06&9.2042e-07&7.6282e-05&1.2340e-05   \\
		$1\slash22$&6.5478e-08& 6.5476e-08&6.6200e-07&4.8351e-07&5.5322e-05&6.5117e-06\\
		\hline
	\end{tabular}
\end{table}

\begin{table}[!ht]
	\centering
	\caption{Example 2: Numerical results by the third-order $(k=3)$ cuboid element in the  family $(r=k-1)$ of $H(\grad \div)$-conforming elements} \label{cube_tab4}
	\begin{tabular}{cccccccc}
		\hline
		$h$ &$\left\|\bm e_h\right\|$&$\left\|\bm e_h\right\|_V$&$\left\|\nabla\cdot\bm e_h\right\|$&$\left\|\nabla\cdot\bm e_h\right\|_U$&$\left\|\nabla(\nabla\cdot\bm e_h)\right\|$& $\left\|(\nabla\times)^2\bm e_h\right\|_W$\\
		\hline
		$1\slash 4$ &6.3209e-07&2.2806e-07 &  2.9623e-06& 8.8580e-07 &7.8540e-05&2.5723e-05 \\
		$1\slash10$&7.6991e-08&1.1031e-08&1.8971e-07  &2.8011e-08&1.2378e-05 & 1.9063e-06\\
		$1\slash16$&2.8833e-08 & 2.5640e-09&4.6416e-08&4.3703e-09&4.8272e-06& 4.7467e-07 \\
		$1\slash22$&1.5050e-08& 9.7038e-10&1.7869e-08&1.2319e-09&2.5519e-06&1.8377e-07\\
		\hline
	\end{tabular}
\end{table}

 \begin{figure}[ht]
  \subfloat[$r=k-1,k=2$ ]{
	\begin{minipage}[c][0.8\width]{
	   0.45\textwidth}
	   \centering
	   \includegraphics[width=1\textwidth]{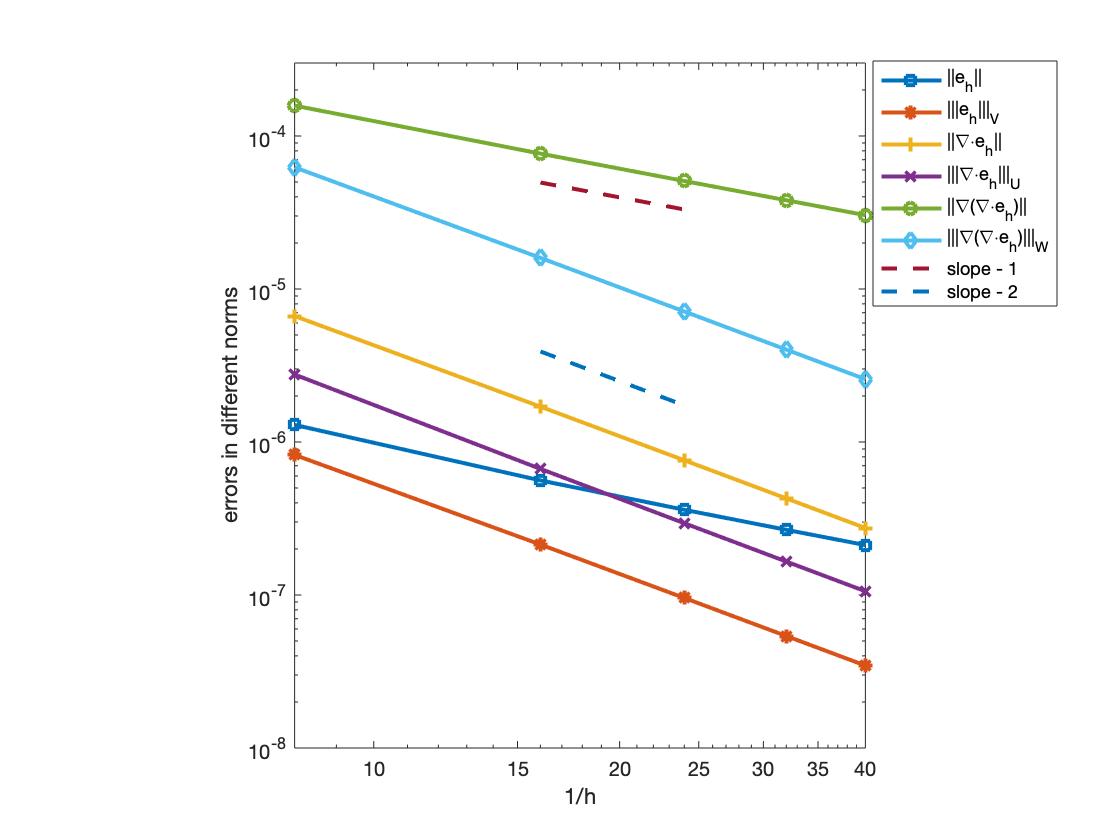}
	\end{minipage}}
 \hfill 	
  \subfloat[$r=k,k=2$ ]{
	\begin{minipage}[c][0.8\width]{
	   0.45\textwidth}
	   \centering
	   \includegraphics[width=1\textwidth]{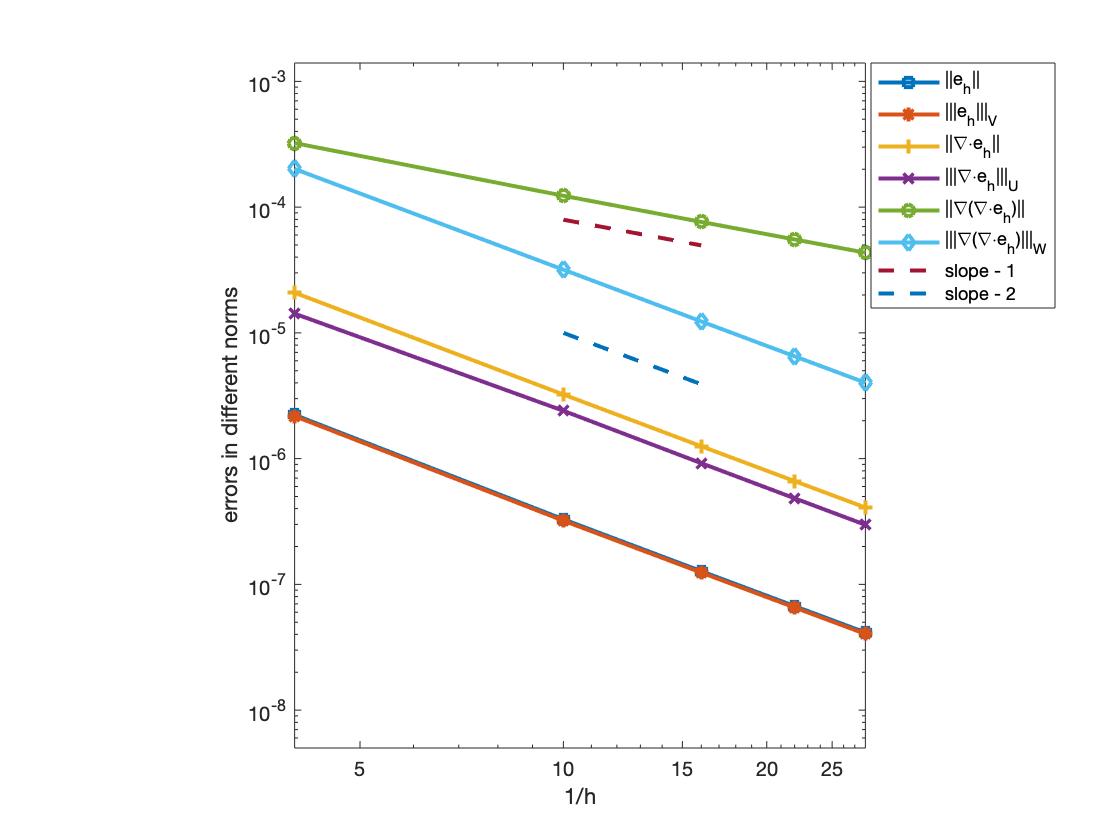}
	\end{minipage}}
	\hfill 
  \subfloat[$r=k+1,k=2$ ]{
	\begin{minipage}[c][0.8\width]{
	   0.45\textwidth}
	   \centering
	   \includegraphics[width=1\textwidth]{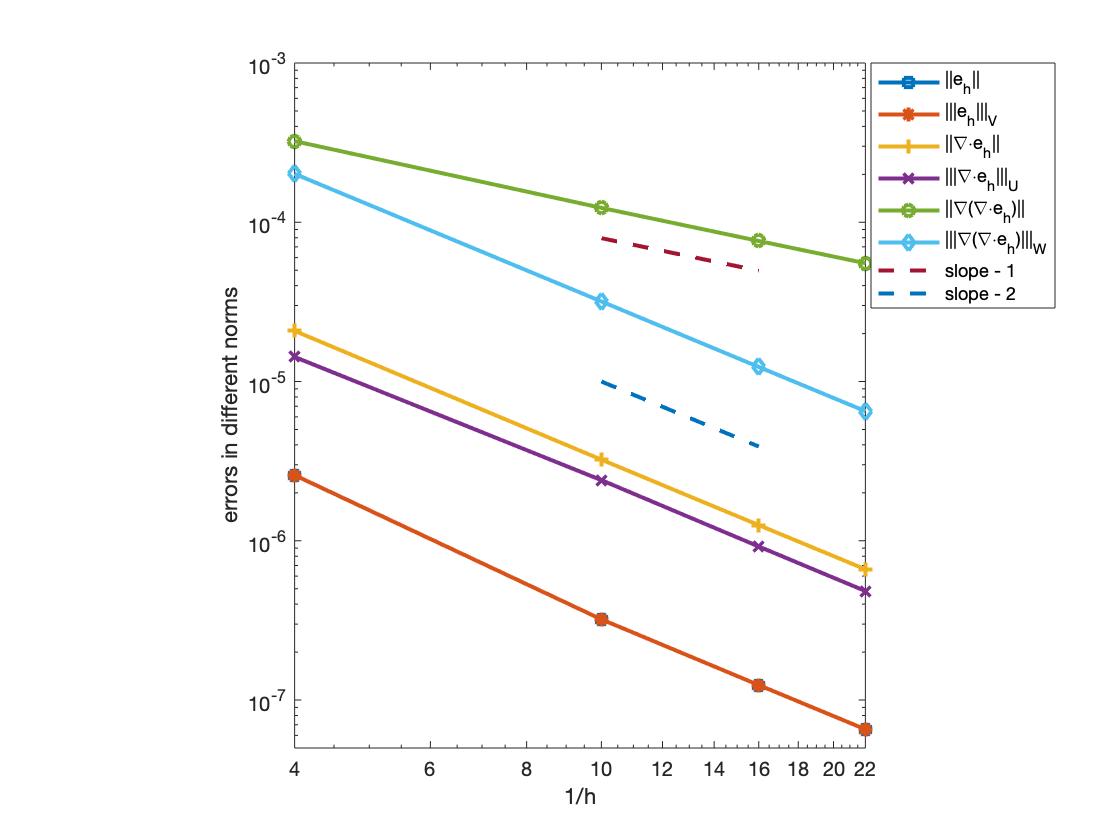}
	\end{minipage}}
	\hfill 
  \subfloat[$r=k-1,k=3$ ]{
	\begin{minipage}[c][0.8\width]{
	   0.45\textwidth}
	   \centering
	   \includegraphics[width=1\textwidth]{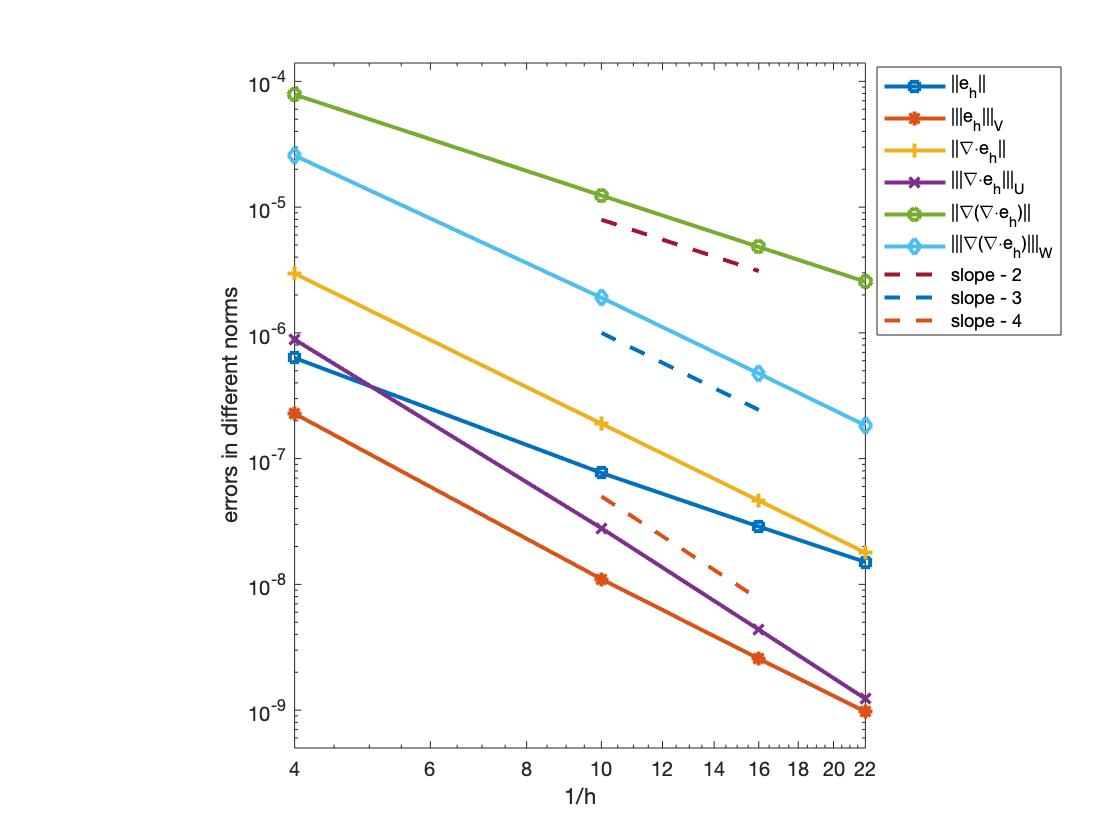}
	\end{minipage}}
 \caption{Error curves in different norms}\label{Fig:errorcurve}
\end{figure}

%\begin{figure}[!h]
%\includegraphics[width=7.5cm]{errorcurve}
%\caption{Error curves in different norms}
%\label{Fig:errorcurve}
%\end{figure}

\end{example}

\section{Conclusion}
In this paper, we constructed conforming finite element de Rham complexes with enhanced smoothness. This naturally leads to grad-div conforming elements in 3D which can be utilized to solve the quad-div problem. The simplest elements in our construction have only 8 and 14 DOFs for a tetrahedron and a cuboid respectively, which makes commercial adoption of the elements feasible.

Since the quad-div problem has not been extensively studied both in mathematical theory and numerical methods, there are still some mysteries about the regularity and the discretization of this problem. However, by the de Rham complex, we relate the grad-div conforming elements to the FEEC. This allows further systematic developments of the new elements and the quad-div problem. We believe that the framework may shed new light into studying this problem.

In the future, we will apply the newly proposed elements to solve practical problems and further investigate the superconvergence phenomena. Moreover, we will construct simple finite element subcomplexes for the following variant of  de Rham complex:
\begin{align*}
\begin{diagram}
\mathbb{R} & \rTo^{\subset} & H^{1}(\Omega) & \rTo^{\nabla} & \bm\Phi(\Omega) & \rTo^{\nabla\times} & \bm H^1(\Omega)& \rTo^{\nabla\cdot} & L^2(\Omega)  & \rTo^{} & 0,
\end{diagram}
\end{align*}
where $\mathbf{\Phi}(\Omega):=\left\{\boldsymbol{v} \in \boldsymbol{L}^{2}(\Omega), \nabla\times\boldsymbol{v} \in \boldsymbol{H}^{1}(\Omega)\right\}$.

\bibliographystyle{plain}
\bibliography{quaddiv-3d-4}{}

\begin{thebibliography}{10}

\bibitem{altan1992structure}
S.~Altan and E.~Aifantis.
\newblock On the structure of the mode {III} crack-tip in gradient elasticity.
\newblock {\em Scripta Metallurgica et Materialia}, 26(2):319--324, 1992.

\bibitem{amrouche1998vector}
C.~Amrouche, C.~Bernardi, M.~Dauge, and V.~Girault.
\newblock Vector potentials in three-dimensional non-smooth domains.
\newblock {\em Mathematical Methods in the Applied Sciences}, 21(9):823--864,
  1998.

\bibitem{argyris1968tuba}
J.~Argyris, I.~Fried, and D.~Scharpf.
\newblock The {TUBA} family of plate elements for the matrix displacement
  method.
\newblock {\em The Aeronautical Journal}, 72(692):701--709, 1968.

\bibitem{arnold2018finite}
D.~Arnold.
\newblock {\em {Finite Element Exterior Calculus}}, volume~93.
\newblock SIAM, 2018.

\bibitem{arnold2006finite}
D.~Arnold, R.~Falk, and R.~Winther.
\newblock Finite element exterior calculus, homological techniques, and
  applications.
\newblock {\em Acta Numerica}, 15:1--155, 2006.

\bibitem{arnold2010finite}
D.~Arnold, R.~Falk, and R.~Winther.
\newblock Finite element exterior calculus: from hodge theory to numerical
  stability.
\newblock {\em Bulletin of the American Mathematical Society}, 47(2):281--354,
  2010.

\bibitem{arnold2014periodic}
D.~Arnold and A.~Logg.
\newblock Periodic table of the finite elements.
\newblock {\em SIAM News}, 47(9):212, 2014.

\bibitem{BrennerSC2019Multigrid100}
S.~Brenner, J.~Cui, and L.~Sung.
\newblock Multigrid methods based on {H}odge decomposition for a quad-curl
  problem.
\newblock {\em Computational Methods in Applied Mathematics}, 19(2):215--232,
  2019.

\bibitem{brenner2008mathematical}
S.~Brenner and R.~Scott.
\newblock {\em {The Mathematical Theory of Finite Element Methods}}, volume~15.
\newblock Springer Science \& Business Media, 2008.

\bibitem{Brenner2017Hodge}
S.~Brenner, J.~Sun, and L.~Sung.
\newblock Hodge decomposition methods for a quad-curl problem on planar
  domains.
\newblock {\em Journal of Computational Science}, 73(2-3):495--513, 2017.

\bibitem{Chen2018Analysis164}
G.~Chen, W.~Qiu, and L.~Xu.
\newblock Analysis of a mixed finite element method for the quad-curl problem.
\newblock {\em arXiv:1811.06724}, 2018.

\bibitem{christiansen2018nodal}
S.~Christiansen, J.~Hu, and K.~Hu.
\newblock Nodal finite element de {R}ham complexes.
\newblock {\em Numerische Mathematik}, 139(2):411--446, 2018.

\bibitem{christiansen2016generalized}
S.~Christiansen and K.~Hu.
\newblock Generalized finite element systems for smooth differential forms and
  stokes' problem.
\newblock {\em Numerische Mathematik}, 140(2):327--371, 2018.

\bibitem{reg-laplace}
A.~Demlow.
\newblock Elliptic problems on polyhedral domains.
\newblock {\em Lecture Notes}, 2016.

\bibitem{falk2013stokes}
R.~Falk and M.~Neilan.
\newblock Stokes complexes and the construction of stable finite elements with
  pointwise mass conservation.
\newblock {\em SIAM Journal on Numerical Analysis}, 51(2):1308--1326, 2013.

\bibitem{fan2019mixed}
R.~Fan, Y.~Liu, and S.~Zhang.
\newblock Mixed schemes for fourth-order div equations.
\newblock {\em Computational Methods in Applied Mathematics}, 19(2):341--357,
  2019.

\bibitem{Girault2012Finite}
V.~Girault and P.~Raviart.
\newblock {\em Finite element methods for {N}avier-{S}tokes equations: theory
  and algorithms}, volume~5.
\newblock Springer Science \& Business Media, 2012.

\bibitem{hiptmair1999canonical}
R.~Hiptmair.
\newblock Canonical construction of finite elements.
\newblock {\em Mathematics of Computation}, 68(228):1325--1346, 1999.

\bibitem{Qingguo2012A}
Q.~Hong, J.~Hu, S.~Shu, and J.~Xu.
\newblock A discontinuous {G}alerkin method for the fourth-order curl problem.
\newblock {\em Journal of Computational Mathematics}, 30(6):565--578, 2012.

\bibitem{HZZcurlcurl2D}
K.~Hu, Q.~Zhang, and Z.~Zhang.
\newblock Simple curl-curl-conforming finite elements in two dimensions.
\newblock {\em arXiv:2004.12507}, 2020.

\bibitem{mindlin1963microstructure}
R.~Mindlin.
\newblock Micro-structure in linear elasticity.
\newblock {\em Arch. Ration. Mech. Anal.}, 16:51--78, 1964.

\bibitem{mindlin1965second}
R.~Mindlin.
\newblock Second gradient of strain and surface-tension in linear elasticity.
\newblock {\em International Journal of Solids and Structures}, 1(4):417--438,
  1965.

\bibitem{Monk2003}
P.~Monk.
\newblock {\em Finite Element Methods for {M}axwell's Equations}.
\newblock Oxford University Press, 2003.

\bibitem{morley1968triangular}
L.~Morley.
\newblock The triangular equilibrium element in the solution of plate bending
  problems.
\newblock {\em The Aeronautical Quarterly}, 19(2):149--169, 1968.

\bibitem{neilan2015discrete}
M.~Neilan.
\newblock {Discrete and conforming smooth de Rham complexes in three
  dimensions}.
\newblock {\em Mathematics of Computation}, 84(295):2059--2081, 2015.

\bibitem{Sun2016A}
J.~Sun.
\newblock A mixed {FEM} for the quad-curl eigenvalue problem.
\newblock {\em Numerische Mathematik}, 132(1):185--200, 2016.

\bibitem{quadcurlWG}
J.~Sun, Q.~Zhang, and Z.~Zhang.
\newblock A curl-conforming weak {G}alerkin method for the quad-curl problem.
\newblock {\em BIT Numerical Mathematics}, 59(4), 2019.

\bibitem{WangC2019Anew101}
C.~Wang, Z.~Sun, and J.~Cui.
\newblock A new error analysis of a mixed finite element method for the
  quad-curl problem.
\newblock {\em Applied Mathematics and Computation}, 349:23--38, 2019.

\bibitem{quad-curl-eig-posterior}
L.~Wang, Q.~Zhang, J.~Sun, and Z.~Zhang.
\newblock A priori and a posterior error estimations of quad-curl eigenvalue
  problems in 2{D}.
\newblock {\em To appear}.

\bibitem{vzenivsek1973polynomial}
A.~{\v{Z}}en{\'\i}{\v{s}}ek.
\newblock Polynomial approximation on tetrahedrons in the finite element
  method.
\newblock {\em Journal of Approximation Theory}, 7(4):334--351, 1973.

\bibitem{WZZelement}
Q.~Zhang, L.~Wang, and Z.~Zhang.
\newblock H($\text{curl}^2$)-conforming finite elements in 2 dimensions and
  applications to the quad-curl problem.
\newblock {\em SIAM Journal on Scientific Computing}, 41(3):A1527 -- A1547,
  2019.

\bibitem{zhang2009family}
S.~Zhang.
\newblock A family of 3{D} continuously differentiable finite elements on
  tetrahedral grids.
\newblock {\em Applied Numerical Mathematics}, 59(1):219--233, 2009.

\bibitem{Zhang2018M2NA}
S.~Zhang.
\newblock Mixed schemes for quad-curl equations.
\newblock {\em Esaim Mathematical Modelling \& Numerical Analysis},
  52(1):147--161, 2018.

\bibitem{Zhang2018Regular162}
S.~Zhang.
\newblock Regular decomposition and a framework of order reduced methods for
  fourth order problems.
\newblock {\em Numerische Mathematik}, 138:241--271, 2018.

\bibitem{Zheng2011A}
B.~Zheng, Q.~Hu, and J.~Xu.
\newblock A nonconforming finite element method for fourth order curl equations
  in $\mathbb{R}^3$.
\newblock {\em Mathematics of Computation}, 80(276):1871--1886, 2011.

\end{thebibliography}
~\newline
\end{document}